\documentclass[12pt,twoside]{article}
\usepackage{mathtools,amssymb,amsthm}
\usepackage[dvipdfmx]{graphicx}
\usepackage{comment}
\usepackage{color}
\usepackage{fancyhdr}
\newtheorem{definition}{Definition}[section]
\newtheorem{theorem}[definition]{Theorem}
\newtheorem{proposition}[definition]{Proposition}
\newtheorem{corollary}[definition]{Corollary}
\newtheorem{lemma}[definition]{Lemma}
\newtheorem{remark}[definition]{Remark}

\numberwithin{equation}{section}

\begin{document}

\allowdisplaybreaks[4]

\title{Existence of self-similar solutions for the surface \\ diffusion flow with nonlinear boundary conditions \\ in the half space}

\author{Tomoro Asai\footnotemark\,\,\ and Yoshihito Kohsaka\footnotemark}

\footnotetext[1]{Graduate School of Science and Engineering, Saitama University, 255 Shimo-Okubo, Sakura-ku, Saitama City, Saitama 338-8570, Japan; 
e-mail:\,edu54808@mail.saitama-u.ac.jp} 
\footnotetext[2]{Graduate School of Maritime Sciences, Kobe University, 5-1-1 Fukaeminami-machi, Higashinada-ku, Kobe 658-0022, Japan; 
e-mail:\,kohsaka@maritime.kobe-u.ac.jp}

\date{\empty}

\maketitle

\pagestyle{fancy}
\fancyhead{}
\fancyhead[LE]{\thepage}
\fancyhead[CE]{T. Asai and Y. Kohsaka}
\fancyhead[RO]{\thepage}
\fancyhead[CO]{Existence of self-similar solutions for the surface diffusion flow}
\renewcommand{\headrulewidth}{0.0pt}
\fancyfoot{}

\vspace*{-0.5cm}

\begin{abstract}
We study the Mullins problem that was proposed by Mullins in 1957 and is one of the models of the thermal grooving by surface diffusion. 
Mathematically, this is the problem of evolving curves in the half space that is governed by the surface diffusion flow with the contact angle 
condition and the no-flux condition on the boundary. The no-flux condition is represented as the equation that the first order derivative 
of the curvature with respect to the arc-length parameter is equal to zero, so that it is the nonlinear boundary condition. 
For this original Mullins problem, we show the existence and the uniqueness of the self-similar solution. 
The self-similar solution is obtained as the mild solution under the smallness assumption on the contact angle.  
%\keywords{Self-similar solution \and Surface diffusion flow \and Nonlinear boundary condition}
%\subclass{35C06 \and 35G31 \and 74N20}
\end{abstract}

%===============
\section{Introduction}\label{sec:essn-Introduction}
%===============
\indent 
This paper studies the existence and uniqueness of the self-similar solution to the Mullins problem that is the problem of evolving curves 
governed by  the surface diffusion flow with the contact angle condition and the no-flux condition in the half space. More precisely, 
let $\{\Gamma(t)\}_{t\ge0}$ be a family of planar curves in the half space $\Omega_+:=\{(x,y)\in\mathbb{R}^2\mid x\ge0\}$. Assume that 
one endpoint of $\Gamma(t)$ touches the boundary $\partial\Omega_+=\{(x,y)\in\mathbb{R}^2\mid x=0\}$ and the motion 
of $\Gamma(t)$ with the initial condition $\Gamma(0)=\Gamma_0$ is governed by
\begin{equation}\label{geometric}
\left\{\begin{array}{l}
V=-\kappa _{ss}\,\,\ \text{on}\,\ \Gamma(t), \\[0.05cm]
\sphericalangle(\Gamma(t),\partial\Omega_+)=\dfrac{\pi}2-\beta\,\,\ \text{at}\,\ \Gamma(t)\cap\partial\Omega_+, \\[0.25cm]
\kappa_s=0\,\,\ \text{at}\,\ \Gamma(t)\cap\partial\Omega_+. 
\end{array}\right.
\end{equation}
Here $V$ and $\kappa$ are the normal velocity and the curvature of $\Gamma(t)$, respectively, $s$ denotes the arc-length parameter 
along $\Gamma(t)$, and $\beta$ is a constant in the interval $(0,\pi/2)$. The first boundary condition at $\Gamma(t)\cap\partial\Omega_+$ 
is the contact angle condition and the second one is the no-flux condition. Note that the no-flux condition is the nonlinear boundary condition 
(for example, see \eqref{eq:161031a4} or \eqref{3rd-bc} below). 
A more detailed description of this system can be found in Mullins~\cite{mul;57}, or in Asai-Giga~\cite{asa;gig;14}. 
The Figure~\ref{fig:thermal} indicates our situation. 

\begin{figure}[htbp]
\begin{center}
\scalebox{0.8}{\includegraphics{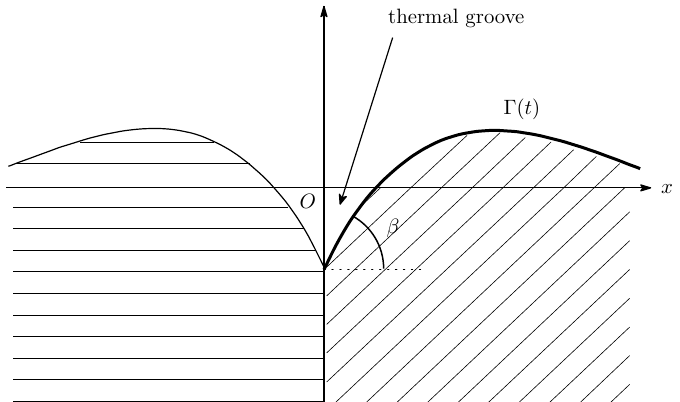}}
\caption{Profile of thermal groove}
\label{fig:thermal}
\end{center}
\end{figure}

The surface diffusion flow was first proposed by Mullins~\cite{mul;57} to model the thermal grooving by surface diffusion. 
Dav\`{\i}--Gurtin~\cite{dav;gur;90} derived the surface diffusion flow from the thermodynamical point of view. 
Taylor--Cahn~\cite{tay;cah;94} characterized the surface diffusion flow as the $H^{-1}$-gradient flow for the area functional. 
This gradient flow structure yields that the evolving closed embedded surface by the surface diffusion flow preserves the enclosed volume 
and decreases the surface area in time. Cahn--Elliott--Novick-Cohen~\cite{cah;ell;nov;96} showed that the interface governed by 
the surface diffusion flow is the singular limit of the zero level set of the solution to the Cahn--Hilliard equation with degenerate mobility. 

The motion of curves by the surface diffusion flow (which is also called the curve diffusion flow in the case of curves) are studied by 
many researchers. With regard to the equilibrium state for the surface diffusion flow and the behavior of a solution around them, for example, 
see Elliott--Garcke~\cite{ell;gar;97} in the case of circles, Miura--Okabe~\cite{miu;oka;21} in the case of multiply covered circles, 
and Garcke--Ito--Kohsaka~\cite{gar;ito;koh;05,gar;ito;koh;08}, Gazwani--McCoy~\cite{gaz;mcc;24}, and Wheeler--Wheeler~\cite{whe;whe;24} 
in the case of a part of straight lines or circles with the right angle condition on the boundary. Concerning the traveling wave, 
we refer to Kagaya--Kohsaka \cite{kag;koh;20} that proved the existence of the non-convex traveling waves, which include sign-changing ones, 
under the setting that the curves with two end points on $\Pi:=\{(x,y)\in\mathbb{R}^2\,|\,x\in\mathbb{R},\,y=0\}$ satisfy the contact angle condition 
and the no-flux condition, and Ogden--Warren~\cite{ogd;war;preprint} that showed the existence of a non-compact, complete translating solution. 
In these papers, it plays a key role to analyze the boundary value problem to the nonlinear third order ordinary differential equation 
for the function associated with the Gauss map. As for the self-similar solution, we introduce Asai--Giga~\cite{asa;gig;14} that 
gives the motivation for this paper.  In \cite{asa;gig;14}, they proved the existence and the stability of the self-similar solution to the problem 
imposing the linearized boundary condition instead of the no-flux condition. For further details, 
see the description below. Recently, Rybka and Wheeler~\cite{ryb;whe;25} classified the solitons, which mean equilibria, self-similar solutions, 
and traveling waves, for the surface diffusion flow of entire graphs over $\mathbb{R}$. In \cite{ryb;whe;25}, they point out that 
all non-trivial solitons to the graphical surface diffusion flow over $\mathbb{R}$ are unbounded. 
Regarding further interesting motion of curves due to the surface diffusion flow, for example, see Giga--Ito~\cite{gig;ito;98,gig;ito;99} 
for pinching or loss of convexity of evolving closed curves, and Elliott--Maier-Paape~\cite{ell;mai;01} for losing a graph. 

Set $\mathbb{R}_+:=[0,\infty)$. Hereafter, we assume that $\Gamma(t)$ is parametrized by the graph of a function $w(x,t)$ for $x\in\mathbb{R}_+$, 
i.e., $\Gamma(t)=\{(x,w(x,t))\,|\,x\in\mathbb{R}_+\}$. 
%
%%%%%%%%%%%%
\begin{comment}

Taking the unit normal vector field $\text{\boldmath$n$}$ to be 
\begin{equation*}
\text{\boldmath$n$}=\biggl(-\frac{w_x}{(1+w_x^2)^{\frac12}},\frac{1}{(1+w_x^2)^{\frac12}}\biggr), 
\end{equation*}
the normal velocity $V$ and the curvature $\kappa$ are given by
\begin{gather*}
V=\frac{w_t}{(1+w_x^2)^{\frac12}},\quad \kappa =\frac{w_{xx}}{(1+w_x^2)^{\frac32}},
\end{gather*}
see, e.g., \cite[Chapter~1]{gig;book;06}. Recalling $\partial_s=(1+w_x^2)^{-\frac12}\partial_x$, 

\end{comment}
%%%%%%%%%%%%
%
Then the geometric problem \eqref{geometric} is represented as 
\begin{align}
&w_t=-\partial_x\biggl\{\dfrac{1}{(1+w_x^2)^{\frac12}}\partial_x\biggl(\dfrac{w_{xx}}{(1+w_x^2)^{\frac32}}\biggr)\biggr\},
\,\,\ x>0,\ t>0, \label{eq:161031a2} \\
&w_x(0,t)=\tan \beta,\,\,\ t>0, \label{eq:161031a3} \\
&\partial_x\biggl(\dfrac{w_{xx}}{(1+w_x^2)^{\frac32}}\biggr)\biggr|_{x=0}=0,\,\,\ t>0. \label{eq:161031a4}
\end{align}
Setting $\Gamma_0:=\{(x,a(x))\,|\,x\in\mathbb{R}_+\}$, which is the parameterization of $\Gamma_0$ 
by the graph of a function $a(x)$ for $x\in\mathbb{R}_+$, 
the initial condition for $w$ is given by 
\begin{equation}\label{eq:161031a5}
w(x,0)=a(x),\,\,\ x\in\mathbb{R}_+.
\end{equation}
The purpose of this paper is to prove the existence and the uniqueness of the self-similar solution to the original Mullins problem 
\eqref{eq:161031a2}--\eqref{eq:161031a5} with $a(x)\equiv0$ under the assumption that the angle 
$\beta>0$ is sufficiently small. 
%and to show the stability of its self-similar solution. 
The definition of the self-similar solution to our problem is as follows. 

%------------------------------
\begin{definition}\label{def:self-similar}
For a function $w_{\ast}:\mathbb{R}_+\times(0,\infty)\to\mathbb{R}$, set 
\[
w_{\ast}^{\lambda}(x,t):=\lambda^{-1}w_{\ast}(\lambda x,\lambda^4t)
\]
for $\lambda>0$. We say that $w_{\ast}$ is a self-similar solution to \eqref{eq:161031a2}--\eqref{eq:161031a4} if $w_{\ast}$ satisfies 
\eqref{eq:161031a2}--\eqref{eq:161031a4} and $w_{\ast}^{\lambda }(x,t)=w_{\ast}(x,t)$.
\end{definition}
%------------------------------

\noindent
We remark that we find the self-similar solution to the original Mullins problem as the solution of the corresponding 
integral equation (see Section \ref{sec:MullinsProblem} for details). 
In the definition of $w_\ast^\lambda$, let $\lambda=t^{-\frac14}$ and set $W(y):=w_\ast(y,1)$. Then $W$ is called 
the profile function and the self-similar solution $w_{\ast}$ is of the form $w_{\ast}(x,t)=t^{\frac14}W(x/t^{\frac14})$. 
If there exists the upper bound of $|W|$ independent of $x$ and $t$, we say that a self-similar solution is bounded. 
We can consider two approaches to prove the existence of a self-similar solution. One is to solve an ordinary differential equation for 
a profile function. For the surface diffusion flow,  it becomes a highly-nonlinear fourth-order ordinary differential equation, so that 
it seems to be difficult to solve it on $\mathbb{R}_+$. Even if we use a function associated with the Gauss map as in \cite{kag;koh;20}, 
it becomes a nonlinear third-order ordinary differential equation with nonlocal terms although the nonlinearity is weakened, 
and so the difficulty still remains to solve it. Another approach is to solve a partial differential equation imposing a homogeneous initial data. 
This approach is introduced initially by Giga--Miyakawa~\cite{gig;miy;89} and developed by Cazenave--Weissler~\cite{caz;wei;98}. 
In our problem, the initial data is zero (see the end of Section \ref{sec:Preliminaries} for details). 
We apply the method of solving a partial differential equation as in \cite{asa;gig;14}. 

Let us comment on the difference from \cite{asa;gig;14} that studied the existence and the stability of the self-similar solution to the problem in which 
\begin{equation}\label{eq:161031a4_linear} 
w_{xxx}(0,t)=0
\end{equation} 
is imposed instead of the no-flux condition \eqref{eq:161031a4} as the second boundary condition. Note that \eqref{eq:161031a4} is equivalent to 
\begin{equation}\label{3rd-bc}
w_{xxx}(0,t)=\frac{3w_x(0,t)}{1+w_x(0,t)^2}(w_{xx}(0,t))^2. 
\end{equation}
From this, it is observed that \eqref{eq:161031a4_linear} is the linearization of \eqref{eq:161031a4} around $w_x\approx 0$ (small slope approximation). 
Due to the nonlinearity of \eqref{3rd-bc}, we find a solution in the space 
\[
B_{\frac{2+\gamma}4}((0,T];BUC^{3+\gamma}(\mathbb{R}_+))\cap C_{\frac{2+\gamma}4}^{\frac{1+\gamma}4}((0,T];BUC^2(\mathbb{R}_+)),
\]
while it is $BUC_{\frac{2+\gamma}4}((0,T];h^{3+\gamma }(\mathbb{R}_+))$ in \cite{asa;gig;14} (see Section \ref{sec:Preliminaries} for 
the definition of the spaces). Note that 
\[
B_{\frac{2+\gamma}4}((0,T];BUC^{3+\gamma}(\mathbb{R}_+))\not\subset
C_{\frac{2+\gamma}4}^{\frac{1+\gamma}4}((0,T];BUC^2(\mathbb{R}_+))
\]
since $B_{\frac{2+\gamma}4}((0,T];BUC^{3+\gamma}(\mathbb{R}_+))$ does not assume any time-regularity on functions, the inclusion, 
in general, does not hold. 
Furthermore, we derive the required estimates in the proof by using the Green function for the Cauchy problem of 
the fourth-order diffusion equation $w_t=-w_{xxxx}$, although the abstract theory of analytic semigroups is applied in \cite{asa;gig;14}. 
One of the advantages in our method is that the well-definedness of the integral equation under our regularity setting becomes 
clearer than \cite{asa;gig;14}. Also, the dependence on $x$ for $w(x,t)$ is explicit and so it can be proved that there exists a negative 
upper bound of $w(0,t)$ as long as $\beta>0$, which guarantees the existence of non-trivial self-similar solution, 
i.e., $w(x,t)\not\equiv0$ (see Remark \ref{rem:existence-selfsimilar}). 
We remark that Koch--Lamm~\cite{koc;lam;12} study the self-similar solutions of the geometric flow including the surface diffusion flow 
as entire graphs over $\mathbb{R}^n$ and use the Green function of the fourth-order diffusion equation on $\mathbb{R}^n$ in their analysis. 
%
%To prove the existence and the uniqueness of the self-similar solution, the problem \eqref{eq:161031a2}--\eqref{eq:161031a3}, \eqref{eq:161031a4_linear} 
%with the zero initial data is solved by using the abstract theory of analytic semigroups, especially due to the maximal regularity result of 
%Da Prato--Grisvard \cite{equationsdevo} and Angenent \cite{nonlinearana}. Once a unique solution is obtained, the rescaled function of its solution 
%also satisfies the same problem. By the uniqueness, it can be concluded that this solution is self-similar. 
%After proving the existence and the uniqueness, a stability of the self-similar solution is discussed. This procedure is done as the limit of 
%the mild solution in an appropriate topology.
%
The Green function as mentioned above is also used to represent the explicit formula of the solution to the initial boundary value problem 
\begin{equation}\label{LP}
\left\{\begin{array}{l}
U_t=-U_{xxxx},\,\,\ x>0,\ t>0, \\[0.05cm]
U_x(0,t)=\tan \beta ,\,\,\ t>0, \\[0.05cm] %\label{eq:161101a7}
U_{xxx}(0,t)=c_{\beta}b(t), \,\,\ t>0, \\[0.05cm] %\label{eq:161101a8}
U(x,0)=0,\,\,\ x\in\mathbb{R}_+,
\end{array}\right.
\end{equation}
where $b(t)$ is a given function and 
\begin{equation}\label{c_beta}
c_\beta:=\frac{3\tan\beta}{1+\tan^2\beta}\,(<3\tan\beta).
\end{equation}
The explicit formula of the solution to the problem in which $U_{xxx}(0,t)=0$ instead of $U_{xxx}(0,t)=c_{\beta}b(t)$ is obtained 
in Martin~\cite{mar;09}. In the case of the solution to the problem \eqref{LP}, the term depending on $b(t)$ is added to the solution 
by Martin. To estimate this additional term in our setting of the space, a function $b$ should lie in 
$C_{\frac{3+\gamma}4}^{\frac{1+\gamma}4}((0,T];\mathbb{R})$ (see Remark \ref{rem:reg_wxx2}). 

This paper is organized as follows. In Section~\ref{sec:Preliminaries}, we give the definition of the function spaces and rewrite the equation 
\eqref{eq:161031a2} as the divergence form. In Section~\ref{sec:LinearProblem} we introduce the explicit formula of the solution to 
the linear problem \eqref{LP} and derive the estimate of it in $B_{\frac{2+\gamma}4}((0,T];BUC^{3+\gamma}(\mathbb{R}_+))\cap 
C_{\frac{2+\gamma}4}^{\frac{1+\gamma}4}((0,T];BUC^2(\mathbb{R}_+))$. Section~\ref{sec:MullinsProblem} is devoted to showing the existence 
of the self-similar solution to the original Mullins problem \eqref{eq:161031a2}--\eqref{eq:161031a5} with $a(x)\equiv0$. 
In Subsection \ref{subsec:Grenn-func} the required estimates for the Green function are derived. In Subsection \ref{subsec:mild-sol}, 
by subtracting the solution $U$ to \eqref{LP} from $w$ to the nonlinear problem, we reduce the problem with inhomogeneous boundary conditions 
to one with the homogeneous condtions. Based on this homogeneous problem, the integral equation corresponding to 
the problem \eqref{eq:161031a2}--\eqref{eq:161031a5} with $a(x)\equiv0$ is obtained. We state the main theorem on the existence 
of a self-similar solution and prove it by constructing the solution to the integral equation (i.e., the mild solution) via a fixed point argument. 
It is also shown that the mild solution is a weak solution. The detailed estimates on the inhomogeneous part of the equation are given 
in the appendix. 

%=======================================
\section{Preliminaries}\label{sec:Preliminaries}
%=======================================
First of all, we introduce some function spaces. 
%For details, see~\cite[Definitions 2.2, 2.3 and 2.4]{asa;gig;14}. 
Let $\Omega$ be a (possibly unbounded) interval in $\mathbb{R}$ and let $BUC^\ell(\Omega)\,(\ell\in \mathbb{N}\cup\{0\})$ be the space 
consisting of $\ell$-times differentiable functions on $\Omega$ whose derivatives up to the order $\ell$ are bounded and uniformly continuous. 
%
%$B(\Omega),$ $UC(\Omega)$, $C^k(\Omega)\,(k\in\mathbb{N}\cup\{0\})$ be the spaces consisting respectively of bounded, uniformly continuous, 
%$k$ times continuously differentiable functions on $\Omega$. 
%Then the spaces $BC(\Omega)$, $BUC(\Omega)$, $BUC^\ell(\Omega)\,(\ell\in \mathbb{N}\cup\{0\})$ and their norms are given by
%
The norm of $BUC^\ell(\Omega)$ is given by
\begin{align*}
%&BC(\Omega):=B(\Omega)\cap C(\Omega), \quad BUC(\Omega):=B(\Omega)\cap UC(\Omega), \\ 
%&BUC^\ell(\Omega):=\{w\in C^\ell(\Omega)\,|\,\partial_x^kw\in BUC(\Omega)\,(k=0,1,\cdots,\ell)\}, \\
&\|w\|_{B(\Omega)}:=\sup_{x\in\Omega}|w(x)|, \quad 
%\|w\|_{BUC(\Omega)}:=\|w\|_{B(\Omega)}, \quad 
\|w\|_{BUC^\ell(\Omega)}:=\sum _{k=0}^{\ell}\|\partial_x^kw\|_{B(\Omega)}.
\end{align*}
For $\ell\in \mathbb{N}\cup\{0\}$ and $\gamma\in(0,1)$ we define $BUC^{\ell+\gamma}(\Omega)$ and its norm as 
\begin{align*}
&BUC^{\gamma}(\Omega):=\biggl\{w\in BUC(\Omega)\,\bigg|\,
[w]_{C^\gamma(\Omega)}:=\sup _{\begin{subarray}{c}x,y\in\Omega \\ x\ne y\end{subarray}}\frac{|w(x)-w(y)|}{|x-y|^{\gamma}}<\infty\biggr\}, \\
&BUC^{\ell+\gamma}(\Omega):=\{w\in BUC^\ell(\Omega)\,|\,\partial_x^\ell w\in BUC^{\gamma}(\Omega)\}, \\[0.05cm]
&\|w\|_{BUC^{\ell+\gamma}(\Omega)}:=\|w\|_{BUC^\ell(\Omega)}+[\partial_x^\ell w]_{C^\gamma(\Omega)}.
\end{align*}
Hereafter, we write $\|\,\cdot\,\|_\infty$ and $[\,\cdot\,]_\gamma$ instead of $\|\,\cdot\,\|_{B(\Omega)}$ and $[\,\cdot\,]_{C^\gamma(\Omega)}$, respectively.  
%
%Hereafter, we write $\|\,\cdot\,\|_\infty$, $\|\,\cdot\,\|_{\ell,\infty}$, $[\,\cdot\,]_\gamma$, $\|\,\cdot\,\|_{\ell+\gamma}$ instead of 
%$\|\,\cdot\,\|_{B(\Omega)}$, $\|\,\cdot\,\|_{BUC^\ell(\Omega)}$, $[\,\cdot\,]_{C^\gamma(\Omega)}$, $\|\,\cdot\,\|_{BUC^{\ell+\gamma}(\Omega)}$. 
%
Let $E$ be a Banach space with the norm $\|\,\cdot\,\|_E$. Then, for $\alpha>0$, $\gamma\in(0,1)$ and $T>0$ we define the spaces 
%$B_\alpha((0,T];E)$ and $C_\alpha^\gamma((0,T];E)$ as 
$B_\alpha((0,T];E)$, $C^\gamma((0,T];E)$, and $C_\alpha^\gamma((0,T];E)$ as 
\begin{align*}
&B_\alpha((0,T];E):=\{w:(0,T]\to E\,|\sup_{t\in(0,T]}t^\alpha\|w(t)\|_E<\infty\}, \\
&C^\gamma((0,T];E):=\biggl\{w\in BC((0,T];E)\,\bigg|
\sup _{\begin{subarray}{c}s,t\in(0,T] \\ s\ne t\end{subarray}}\frac{\|w(t)-w(s)\|_E}{|t-s|^{\gamma}}<\infty\biggr\}, \\
&C_\alpha^\gamma((0,T];E):=\biggl\{w\in B_{\alpha-\gamma}((0,T];E)\cap C^\gamma((0,T];E)\,\bigg| \\
&\hspace*{92pt} 
[w]_{C_\alpha^\gamma((0,T];E)}:=\sup _{\begin{subarray}{c}s,t\in(0,T] \\ s<t\end{subarray}}s^\alpha\,\frac{\|w(t)-w(s)\|_E}{(t-s)^{\gamma}}<\infty\biggr\}
\end{align*}
with the norms 
\begin{align*}
&\|w\|_{B_\alpha((0,T];E)}:=\sup_{t\in(0,T]}t^\alpha\|w(t)\|_E, \\
&\|w\|_{C_\alpha^\gamma((0,T];E)}:=\|w\|_{B_{\alpha-\gamma}((0,T];E)}+[w]_{C_\alpha^\gamma((0,T];E)},
\end{align*}
where $BC((0,T];E)$ is the space consisting of bounded and continuous functions $w:(0,T]\to E$. 
For more details of the spaces above, see \cite[Chapter 0 and Chapter 4]{lun;95;book}. 

Let $\gamma\in(0,1)$. For a fixed $T>0$ and $\ell\in\{0,1,2,3\}$, we introduce the equivalent norms of $BUC^\ell(\mathbb{R}_+)$ and 
$BUC^{\ell+\gamma}(\mathbb{R}_+)$ by
\begin{align*}
&\|w\|_{BUC^\ell(\mathbb{R}_+)}=\sum_{k=0}^\ell\frac1{T^{\frac{\ell-k}4}}\|\partial_x^kw\|_{\infty}, \\
&\|w\|_{BUC^{\ell+\gamma}(\mathbb{R}_+)}=\sum_{k=0}^\ell\frac{1}{T^{\frac{\ell-k+\gamma}4}}\|\partial_x^kw\|_{\infty}+[\partial_x^\ell w]_\gamma.
\end{align*}
Then the norms of $B_{\frac{2+\gamma}4}((0,T];BUC^{\ell+\gamma}(\mathbb{R}_+))$ and 
$C_{\frac{2+\gamma}4}^{\frac{1+\gamma}4}((0,T];BUC^2(\mathbb{R}_+))$ are given by 
\begin{align*}
%&\hspace*{-12pt}
\|w\|_{B_{\frac{2+\gamma}4}((0,T];BUC^{\ell+\gamma}(\mathbb{R}_+))} %\\
=&\,\sup_{t\in(0,T]}t^{\frac{2+\gamma}4}\|w(t)\|_{BUC^{\ell+\gamma}(\mathbb{R}_+)} \\
=&\,\sum_{k=0}^\ell\frac{1}{T^{\frac{\ell-k+\gamma}4}}\sup _{t\in(0,T]}t^{\frac{2+\gamma}4}\|\partial_x^kw(t)\|_{\infty}
+\sup _{t\in(0,T]}t^{\frac{2+\gamma}4}[\partial_x^\ell w(t)]_{\gamma}, \\
%&\hspace*{-12pt}
\|w\|_{C_{\frac{2+\gamma}4}^{\frac{1+\gamma}4}((0,T];BUC^2(\mathbb{R}_+))} %\\
=&\,\|w\|_{B_{\frac{2+\gamma}4-\frac{1+\gamma}4}((0,T];BUC^2(\mathbb{R}_+))}
+[w]_{C_{\frac{2+\gamma}4}^{\frac{1+\gamma}4}((0,T];BUC^2(\mathbb{R}_+))} \\
=&\,\sum_{k=0}^2\frac1{T^{\frac{2-k}4}}\sup_{t\in(0,T]}t^{\frac14}\|\partial_x^kw(t)\|_{\infty} \\
&\,+\sum_{k=0}^2\frac1{T^{\frac{2-k}4}}\sup _{\begin{subarray}{c}s,t\in(0,T] \\ t>s\end{subarray}}s^{\frac{2+\gamma}4}\,
\frac{\|\partial_x^kw(t)-\partial_x^kw(s)\|_\infty}{(t-s)^{\frac{1+\gamma}4}}.
\end{align*}
Our motivation is to construct a global-in-time solution. To this end, we seek estimates with constants independent of $T$. 
By choosing these norms, we can cancel out the powers of $t$ that arise in our estimates. Thus these definitions are well-suited for our purpose. 

We consider the problem \eqref{eq:161031a2}--\eqref{eq:161031a5} with $a(x)\equiv0$, which corresponds to the initial boundary value problem 
for a bounded self-similar solution. To analyze this, let us transform \eqref{eq:161031a2} and \eqref{eq:161031a4}. 
Taking account of 
\begin{equation}\label{formula_1}
\partial_x\biggl(\dfrac{w_{xx}}{(1+w_x^2)^{\frac32}}\biggr)
=\frac{w_{xxx}}{(1+w_x^2)^{\frac32}}-\frac{3w_x(w_{xx})^2}{(1+w_x^2)^{\frac52}},
\end{equation}
we have 
\[
\partial_x\biggl\{\dfrac1{(1+w_x^2)^{\frac12}}\partial_x\biggl(\dfrac{w_{xx}}{(1+w_x^2)^{\frac32}}\biggr)\biggr\}
=\partial_x\biggl(\frac{w_{xxx}}{(1+w_x^2)^2}-\frac{3w_x(w_{xx})^2}{(1+w_x^2)^3}\biggr).
\]
This implies that 
\begin{align}
&\text{(R.H.S. of \eqref{eq:161031a2})} \notag \\
%=&\,-w_{xxxx}-\partial_x\biggl(\frac{w_{xxx}}{(1+w_x^2)^2}-\frac{3w_x(w_{xx})^2}{(1+w_x^2)^3}-w_{xxx}\biggr) \\
&=-w_{xxxx}-\partial_x\biggl\{\biggl(\frac1{(1+w_x^2)^2}-1\biggr)w_{xxx}-\frac{3w_x(w_{xx})^2}{(1+w_x^2)^3}\biggr\}, \label{transformed_RHS}
\end{align}
so that \eqref{eq:161031a2} is represented as 
\begin{equation}\label{transformed_eq}
w_t=-w_{xxxx}-\partial_x\Xi(w_x,w_{xx},w_{xxx}),
\end{equation}
where 
\begin{equation}\label{def_Xi}
\Xi(w_x,w_{xx},w_{xxx}):=\biggl(\frac1{(1+w_x^2)^2}-1\biggr)w_{xxx}-\frac{3w_x(w_{xx})^2}{(1+w_x^2)^3}. 
\end{equation}
Also, by virtue of \eqref{formula_1}, we see that \eqref{eq:161031a4} is equivalent to \eqref{3rd-bc}. 
Since $w_x(0,t)=\tan\beta$ by \eqref{eq:161031a3}, the second boundary condition \eqref{eq:161031a4} is represented as 
\begin{equation}\label{transformed_2ndbc}
w_{xxx}(0,t)=c_\beta(w_{xx}(0,t))^2,
\end{equation}
where $c_\beta$ is given by \eqref{c_beta}. Consequently, we arrive at the following initial boundary value problem 
for a bounded self-similar solution. 
\begin{equation}\label{SSP}
\left\{\begin{array}{l}
w_t=-w_{xxxx}-\partial_x\Xi(w_x,w_{xx},w_{xxx}),\,\,\ x>0,\ t>0, \\
w_x(0,t)=\tan\beta,\,\,\ t>0, \\
w_{xxx}(0,t)=c_\beta(w_{xx}(0,t))^2,\,\,\ t>0, \\
w(x,0)=0,\,\,\ x\in\mathbb{R}_+.
\end{array}\right.
\end{equation}
In the following sections, our aim is to establish the existence and the uniqueness of a global-in-time solution to \eqref{SSP} 
provided that $\beta$ is sufficiently small. Indeed, let $w$ be a solution to \eqref{SSP} and let $w_\ast$ be the self-similar solution 
as in Definition \ref{def:self-similar}. Then $w_\ast$ takes the form $w_{\ast}(x,t)=t^{\frac14}W(x/t^{\frac14})$, where 
$W(y):=w_\ast(y,1)$. If $|W|$ is bounded, $w_\ast$ satisfies the initial condition $w_\ast(x,0)\equiv0$. 
Thus $w_\ast$ is a solution to \eqref{SSP}. Given the uniqueness of the solution to \eqref{SSP}, we see that 
\[
w(x,t)=w_\ast(x,t)=t^{\frac14}W\Bigl(\frac{x}{t^{\frac14}}\Bigr).
\]
That is, the existence of a bounded self-similar solution follows from the unique solvability of \eqref{SSP}. 

%------------------------------
\begin{remark}\label{rem:reg_wxx2}
For $0<s<t\le T$, we have 
\begin{align*}
&\hspace*{-10pt}
|(w_{xx}(0,t))^2-(w_{xx}(0,s))^2| \\[0.05cm]
=&\,|w_{xx}(0,t)+w_{xx}(0,s)||w_{xx}(0,t)-w_{xx}(0,s)| \\
\le&\,2\sup_{t\in(0,T]}t^{\frac14}|w_{xx}(0,t)| %\\
%&\,
\biggl(\sup_{\begin{subarray}{c}s,t\in(0,T] \\ s<t\end{subarray}}
s^{\frac{2+\gamma}4}\,\frac{|w_{xx}(0,t)-w_{xx}(0,s)|}{|t-s|^{\frac{1+\gamma}4}}\biggr)
s^{-\frac{3+\gamma}4}(t-s)^{\frac{1+\gamma}4}
\end{align*}
Thus, if $w_{xx}(0,\,\cdot\,)\in C_{\frac{2+\gamma}4}^{\frac{1+\gamma}4}((0,T];\mathbb{R})$, then 
$(w_{xx}(0,\,\cdot\,))^2\in C_{\frac{3+\gamma}4}^{\frac{1+\gamma}4}((0,T];\mathbb{R})$. 
\end{remark}
%------------------------------

%%%%%%%%%%%%%%%%%%%%%%%%
\begin{comment}
%------------------------------
\begin{remark}\label{rem:reg_wxx2}
When we have a self-similar solution $w_{\ast}$ to \eqref{eq:161031a2}--\eqref{eq:161031a4}, $w_{\ast}$ is represented as 
$w_{\ast}(x,t)=t^{\frac14}W(x/t^{\frac14})$ for some function $W$. If the second derivative of $W$ exists and is right-continuous at $x=0$, 
we obtain 
\[
(w_{\ast}_{xx}(0,t))^2=(t^{-\frac14}W''(0))^2=t^{-\frac12}(W''(0))^2\,(=:\widetilde{b}(t)). 
\]
Since $t^{\frac12}\,\widetilde{b}(t)$ is bounded and $t^{\frac12+\mu}\,\widetilde{b}(t)$ is H\"older continuous with exponent $\mu$, 
referring Lunardi \cite[p.123]{lun;95;book}, we see that $\widetilde{b}\in BUC_{\frac12+\mu}^\mu((0,T];\mathbb{R})$. Thus it is expected that 
$(w_{xx}(0,t))^2\in BUC_{\frac12+\mu}^\mu((0,T];\mathbb{R})$ in \eqref{transformed_2ndbc}. 
\end{remark}
%------------------------------
\end{comment}
%%%%%%%%%%%%%%%%%%%%%%%%

%=======================================
\section{Linear problem}\label{sec:LinearProblem}
%=======================================

Let us derive the representation of the solution $U$ to the linear problem \eqref{LP}. 
%In this section, we consider \eqref{LP} with $b(t)$ instead of $b^w(t)$ in the second boundary condition. 
For $x\in\mathbb{R}$ and $t>0$, define the function $G(x,t)$ as 
\begin{equation}\label{GreenFunc}
G(x,t):=\frac1{2\pi}\int_{\mathbb{R}}e^{-\xi^4t+\textbf{i}x\xi}\,d\xi=\frac1{\pi}\int_0^\infty e^{-\xi^4t}\cos(x\xi)\,d\xi.
\end{equation}
Then $G(x,t)$ is the Green function for the fourth-order diffusion equation on $\mathbb{R}$. 
Applying an argument similar to \cite[the proof of Theorem 3.2]{cui;01} or \cite[Appendix B]{ish;koh;miy;sak;25}, 
we obtain the following lemma.
%
%---------------------------------------------
\begin{lemma}\label{lem:Cui}
There exist $C,\nu>0$ such that for all $\ell,k\in\mathbb{Z}_+:=\{m\in\mathbb{Z}\,|\,m\ge0\}$ and $(x,t)\in\mathbb{R}\times(0,\infty)$
\begin{align*}
|\partial_t^\ell\partial_x^k G(x,t)|
\le C\bigg\{1+\biggl(\frac{|x|}{t^{\frac14}}\biggr)^{\frac{4\ell+k}3}\biggr\}\,t^{-\frac{4\ell+k+1}4}e^{-\nu t^{-\frac13}|x|^{\frac43}}.
\end{align*}
\end{lemma}
%---------------------------------------------
%
\noindent
Note that this lemma yields that $G(x,t)$ is a rapidly decreasing function with respect to $x$ on $\mathbb{R}$. 

We define the function $K(x,y,t)$ as 
\begin{align}
K(x,y,t):=&\,G(x-y,t)+G(x+y,t) \notag \\
=&\,\frac1{\pi}\int_0^\infty e^{-\xi^4t}\bigl\{\cos((x-y)\xi)+\cos((x+y)\xi)\bigr\}\,d\xi \label{GreenTypeK}
\end{align}
for $x,y\in\mathbb{R}_+$ and $t>0$. Note that $K(x,y,t)$ is the Green function for the fourth-order diffusion equation 
on $\mathbb{R}_+$ with the Neumann-type boundary conditions, that is, $U_x(x,0)=U_{xxx}(x,0)=0$. 
Using $K(x,y,t)$, the solution $U(x,t)$ to the linear problem \eqref{LP} is represented as 
\begin{equation}\label{U_LP}
U(x,t)=U_1(x,t)\tan\beta+c_\beta U_2(x,t),
\end{equation}
where 
\[
U_1(x,t)=\int_0^t\partial_y^2K(x,0,t-\tau)\,d\tau, \quad 
U_2(x,t)=\int_0^tK(x,0,t-\tau)b(\tau)\,d\tau.
\]

Let us verify that $U$ given by \eqref{U_LP} is the solution to the linear problem \eqref{LP}. 
Since Lemma \ref{lem:Cui} implies 
\[
\lim_{\tau\to t^-}\partial_y^2K(x,0,t-\tau)=\lim_{\tau\to t^-}K(x,0,t-\tau)=0,
\]
this fact and the definition of $K$ yield that $U$ satisfies $U_t=-U_{xxxx}$. In order to verify that $U$ satisfies the boundary conditions 
and the initial condition for \eqref{LP}, we derive the expression for the derivatives of $U_i(i=1,2)$ with respect to $x$ and establish 
the necessary inequalities. By the definition of $K(x,y,t)$ and a simple calculation, we have
\begin{align}
U_1(x,t)=&\,-\frac2{\pi}\int_0^\infty\frac{1-e^{-\xi^4t}}{\xi^2}\cos(x\xi)\,d\xi, \label{int_2ndD_K} \\
U_2(x,t)=&\,\frac2{\pi}\int_0^t\int_0^\infty b(\tau)e^{-\xi^4(t-\tau)}\cos(x\xi)\,d\xi d\tau. \label{int_Kb} 
\end{align}
Using an argument based on complex integrals, it follows that for $x\in\mathbb{R}_+$ 
\[
\int_0^\infty\frac{\sin(x\xi)}{\xi}\,d\xi=\frac{\pi}2, 
\]
so that we obtain
\begin{equation}\label{1stD_U1}
\partial_xU_1(x,t)
=\frac2{\pi}\int_0^\infty\frac{1-e^{-\xi^4t}}{\xi}\sin(x\xi)\,d\xi
=1-\frac2{\pi}\int_0^\infty\frac{e^{-\xi^4t}}{\xi}\sin(x\xi)\,d\xi.
\end{equation}
This implies that
\begin{align}
\partial_x^2U_1(x,t)=&\,-\frac2{\pi}\int_0^\infty e^{-\xi^4t}\cos(x\xi)\,d\xi, \label{2ndD_U1} \\
\partial_x^3U_1(x,t)=&\,\frac2{\pi}\int_0^\infty \xi e^{-\xi^4t}\sin(x\xi)\,d\xi. \label{3rdD_U1}
\end{align}
Regarding $U_2$, we see that for $k\in\{1,2,3\}$
\begin{equation}\label{kthD_U2_1}
\partial_x^kU_2(x,t)
=\frac2{\pi}\int_0^t\int_0^\infty b(\tau)\xi^ke^{-\xi^4(t-\tau)}\cos\Bigl(x\xi+\frac{k\pi}2\Bigr)\,d\xi d\tau.
\end{equation}
Furthermore, transforming \eqref{kthD_U2_1} into 
\begin{align*}
\partial_x^kU_2(x,t)
=&\,\frac2{\pi}\biggl\{b(t)\int_0^t\int_0^{\infty}\xi^ke^{-\zeta^4(t-\tau)}\cos\Bigl(x\xi+\frac{k\pi}2\Bigr)\,d\xi d\tau \\
&+\int_0^t\int_0^{\infty}(b(\tau)-b(t))\xi^ke^{-\xi^4(t-\tau)}\cos\Bigl(x\xi+\frac{k\pi}2\Bigr)\,d\xi d\tau\biggr\}, 
\end{align*}
and taking account of 
\begin{align*}
&\hspace*{-8pt}
\int_0^t\int_0^{\infty}\xi^ke^{-\xi^4(t-\tau)}\cos\Bigl(x\xi+\frac{k\pi}2\Bigr)\,d\xi d\tau \\
%=&\,\int_0^{\infty}\int_0^t\xi^ke^{-\xi^4(t-\tau)}\cos\Bigl(x\xi+\frac{k\pi}2\Bigr)\,d\tau d\xi \\
%=&\,\int_0^\infty\Bigl(\int_0^te^{-\xi^4(t-\tau)}\,d\tau\Bigr)\,\xi^k\cos\Bigl(x\xi+\frac{k\pi}2\Bigr)\,d\xi \\
=&\,\int_0^\infty\frac{1-e^{-\xi^4t}}{\xi^4}\,\xi^k\cos\Bigl(x\xi+\frac{k\pi}2\Bigr)\,d\xi,
\end{align*}
we arrive at 
\begin{align}\label{kthD_U2_2}
\partial_x^kU_2(x,t)
=&\,\frac2{\pi}\biggl\{b(t)\int_0^\infty\frac{1-e^{-\xi^4t}}{\xi^4}\,\xi^k\cos\Bigl(x\xi+\frac{k\pi}2\Bigr)\,d\xi \notag \\
&+\int_0^t\int_0^{\infty}(b(\tau)-b(t))\xi^ke^{-\xi^4(t-\tau)}\cos\Bigl(x\xi+\frac{k\pi}2\Bigr)\,d\xi d\tau\biggr\}.
\end{align}
%\begin{align}
%\partial_x^3U_2(x,t)
%=&\,b(t)\partial_xU_1(x,t)+\frac2{\pi}\int_0^t\int_0^{\infty}\{b(\tau)-b(t)\}\,\xi^3e^{-\xi^4(t-\tau)}\sin(x\xi)\,d\tau d\xi, \label{kthD_U2_2}
%\end{align}
As a preliminary consideration, we state the following lemma. 
%
%------------------------------
\begin{lemma}\label{lem:integral-est}
Let $b\in C_{\frac{3+\gamma}4}^{\frac{1+\gamma}4}((0,T];\mathbb{R})$ with $\gamma\in(0,1)$. 
For $\sigma,t>0$, $p>-1$, $0\le q_1<3$, and $0\le q_2\le4$, it holds that 
\begin{align}
&\int_0^{\infty}\xi^pe^{-\xi^4\sigma}\,d\xi
%=&\,\int_0^{\infty}\eta^{\frac{p}4}\sigma^{-\frac{p}4}e^{-\eta}\cdot\frac14\eta^{-\frac34}\sigma^{-\frac14}\,d\eta \quad
%(\eta=\xi^4\sigma) \notag \\
%=&\,\frac14\sigma^{-\frac{p+1}4}\int_0^{\infty}\eta^{\frac{p+1}4-1}e^{-\eta}\,d\eta \notag \\
=\frac14\,\Gamma\Bigl(\frac{p+1}4\Bigr)\sigma^{-\frac{p+1}4}, \label{ineq_1} \\
&\int_0^t\int_0^{\infty}|b(\tau)|\,\xi^{q_1}e^{-\xi^4(t-\tau)}\,d\xi d\tau \notag \\
&\qquad\le\frac14\,\Gamma\Bigl(\frac{q_1+1}4\Bigr)B\Bigl(\frac12,\frac{3-q_1}4\Bigr)
\|b\|_{B_{\frac12}((0,T];\mathbb{R})}t^{\frac{1-q_1}4}, \label{ineq_2} \\
&\int_0^t\int_0^{\infty}|b(\tau)-b(t)|\,\xi^{q_2}e^{-\xi^4(t-\tau)}\,d\xi d\tau \notag \\
&\qquad\le\frac14\,\Gamma\Bigl(\frac{q_2+1}4\Bigr)B\Bigl(\frac{1-\gamma}4,\frac{\gamma+4-q_2}4\Bigr)
\|b\|_{C_{\frac{3+\gamma}4}^{\frac{1+\gamma}4}((0,T];\mathbb{R})}t^{\frac{1-q_2}4}, \label{ineq_3}
\end{align}
where $\Gamma(\,\cdot\,)$ and $B(\,\cdot\,,\,\cdot\,)$ are the Gamma function and the Beta function, respectively. That is, 
\[
\Gamma(p)=\int_0^\infty\eta^{p-1}e^{-\eta}\,d\eta, \quad 
B(\alpha,\beta)=\int_0^1\sigma^{\alpha-1}(1-\sigma)^{\beta-1}\,d\sigma
\] 
for $p>0$ and $\alpha,\beta>0$. 
\end{lemma}
%------------------------------
%
\begin{proof}
 \eqref{ineq_1} follows from the change of variable by $\eta=\xi^4\sigma$. Concerning \eqref{ineq_2} and \eqref{ineq_3}, 
the change of variable by $\eta=\xi^4(t-\tau)$ and $b\in C_{\frac{3+\gamma}4}^{\frac{1+\gamma}4}((0,T];\mathbb{R})$ 
yield that  
\begin{align*}
&\hspace*{-8pt}
\int_0^t\int_0^{\infty}|b(\tau)|\,\xi^{q_1}e^{-\xi^4(t-\tau)}\,d\xi d\tau \\
%=&\,\frac14\int_0^{\infty}\int_0^t|b(\tau)|\,(t-\tau)^{-\frac{q_1+1}4}\eta^{\frac{q_1+1}4-1}e^{-\eta}\,d\tau d\eta \quad
%(\eta=\xi^4(t-\tau)) \\
\le&\,\frac14\,\Gamma\Bigl(\frac{q_1+1}4\Bigr)\|b\|_{B_{\frac12}((0,T];\mathbb{R})}
\int_0^t(t-\tau)^{-\frac{q_1+1}4}\tau^{-\frac12}\,d\tau \\
%=&\,\frac14\,\Gamma\Bigl(\frac{q_1+1}4\Bigr)\|b\|_{B_{\frac12}((0,T];\mathbb{R})}t^{\frac{1-q_1}4}
%\int_0^1(1-\sigma)^{-\frac{q_1+1}4}\sigma^{-\frac12}\,d\tau \\
=&\,\frac14\,\Gamma\Bigl(\frac{q_1+1}4\Bigr)B\Bigl(\frac12,\frac{3-q_1}4\Bigr)\|b\|_{B_{\frac12}((0,T];\mathbb{R})}t^{\frac{1-q_1}4}, \\
&\hspace*{-8pt}
\int_0^t\int_0^{\infty}|b(\tau)-b(t)|\,\xi^{q_2}e^{-\xi^4(t-\tau)}\,d\xi d\tau \\
%=&\,\frac14\int_0^{\infty}\int_0^t|b(\tau)-b(t)|\,(t-\tau)^{-\frac{q_2+1}4}\eta^{\frac{q_2+1}4-1}e^{-\eta}\,d\tau d\eta \quad
%(\eta=\xi^4(t-\tau)) \\
\le&\,\frac14\,\Gamma\Bigl(\frac{q_2+1}4\Bigr)\|b\|_{C_{\frac{3+\gamma}4}^{\frac{1+\gamma}4}((0,T];\mathbb{R})}
\int_0^t(t-\tau)^{\frac{\gamma-q_2}4}\tau^{-\frac{3+\gamma}4}\,d\tau \\
%=&\,\frac14\,\Gamma\Bigl(\frac{q_2+1}4\Bigr)\|b\|_{C_{\frac{3+\gamma}4}^{\frac{1+\gamma}4}((0,T];\mathbb{R})}t^{\frac{1-q_2}4}
%\int_0^1(1-\sigma)^{\frac{\gamma-q_2}4}\sigma^{-\frac{3+\gamma}4}\,d\tau \\
=&\,\frac14\,\Gamma\Bigl(\frac{q_2+1}4\Bigr)B\Bigl(\frac{1-\gamma}4,\frac{\gamma+4-q_2}4\Bigr)
\|b\|_{C_{\frac{3+\gamma}4}^{\frac{1+\gamma}4}((0,T];\mathbb{R})}t^{\frac{1-q_2}4}.
\end{align*}
This completes the proof. 
\end{proof}
\noindent
Lemma \ref{lem:integral-est} gives the following lemma.

%------------------------------
\begin{lemma}\label{lem:limit}
Let $U_i(i=1,2)$ be given by \eqref{int_2ndD_K} and \eqref{int_Kb}. Then it holds that for each $t>0$
\begin{align}
&\lim_{x\to0^+}\partial_xU_1(x,t)=1, \quad \lim_{x\to0^+}\partial_xU_2(x,t)=0, \label{1stD_limit} \\
&\lim_{x\to0^+}\partial_x^3U_1(x,t)=0, \quad \lim_{x\to0^+}\partial_x^3U_2(x,t)=b(t), \label{3rdD_limit}
\end{align}
and for $x\in\mathbb{R}_+$
\begin{equation}\label{t-0_limit}
\lim_{t\to0^+}U_1(x,t)=0, \quad \lim_{t\to0^+}U_2(x,t)=0. 
\end{equation}
\end{lemma}
%------------------------------

\begin{proof}
We first prove \eqref{1stD_limit} and \eqref{3rdD_limit}. Let $x>0$. Using \eqref{1stD_U1}, \eqref{ineq_1} with $p=0$ and $\sigma=t$, 
\eqref{kthD_U2_1} with $k=1$, and \eqref{ineq_2} with $q_1=2$, we have 
\begin{align*}
&|\partial_xU_1(x,t)-1|\le\frac{2x}{\pi}\int_0^\infty e^{-\xi^4t}\,\biggl|\frac{\sin(x\xi)}{x\xi}\biggr|\,d\xi
\le\frac{x}{2\pi}\,\Gamma\Bigl(\frac14\Bigr)t^{-\frac14}, \\
&|\partial_xU_2(x,t)|\le\frac{2x}{\pi}\int_0^t\int_0^\infty|b(\tau)|\xi^2e^{-\xi^4(t-\tau)}\,\biggl|\frac{\sin(x\xi)}{x\xi}\biggr|\,d\xi \\
&\hspace*{57pt}\le\frac{x}{2\pi}\,\Gamma\Bigl(\frac34\Bigr)B\Bigl(\frac12,\frac14\Bigr)\|b\|_{B_{\frac12}((0,T];\mathbb{R})}t^{-\frac14}
\end{align*}
for each $t>0$. This implies \eqref{1stD_limit}. As for the third derivative, it follows from \eqref{3rdD_U1} and \eqref{ineq_1} with 
$p=2$ and $\sigma=t$ that for each $t>0$
\[
|\partial_x^3U_1(x,t)|\le\frac{2x}{\pi}\int_0^\infty\xi^2e^{-\xi^4t}\,\biggl|\frac{\sin(x\xi)}{x\xi}\biggr|\,d\xi
\le\frac{x}{2\pi}\,\Gamma\Bigl(\frac34\Bigr)t^{-\frac34}.
\]
This gives the first limit of \eqref{3rdD_limit}. Furthermore, considering \eqref{kthD_U2_2} with $k=3$ and recalling \eqref{1stD_U1}, 
we find
\begin{align}\label{3rdD_U2}
\partial_x^3U_2(x,t)
=&\,b(t)\partial_xU_1(x,t) \notag \\
&\,+\frac2{\pi}\int_0^t\int_0^{\infty}(b(\tau)-b(t))\xi^3e^{-\xi^4(t-\tau)}\sin(x\xi)\,d\xi d\tau. 
\end{align}
Since \eqref{ineq_3} with $q_2=4$ yields that 
\begin{align*}
&|(\text{The second term of $\partial_x^3U_2(x,t)$})| \\
&\le\,\frac{2x}{\pi}\int_0^t\int_0^{\infty}|b(\tau)-b(t)|\xi^4e^{-\xi^4(t-\tau)}\,\biggl|\frac{\sin(x\xi)}{x\xi}\biggr|\,d\xi d\tau \\
&\le\,\frac{x}{2\pi}\,\Gamma\Bigl(\frac54\Bigr)B\Bigl(\frac{1-\gamma}4,\frac{\gamma}4\Bigr)
\|b\|_{C_{\frac{3+\gamma}4}^{\frac{1+\gamma}4}((0,T];\mathbb{R})}t^{-\frac34}
\end{align*}
for each $t>0$, this and the first limit of \eqref{1stD_limit} imply the second limit of \eqref{3rdD_limit}. 

Let us prove \eqref{t-0_limit}. As for $U_1$, \eqref{int_2ndD_K} gives 
\[
|U_1(x,t)|\le\frac2{\pi}\int_0^{\infty}\frac{1-e^{-\xi^4t}}{\xi^2}\,d\xi. 
\]
Applying the change of variable by $\eta=\xi^4t$ and the integration by parts, we have 
\begin{align*}
\int_0^{\infty}\frac{1-e^{-\xi^4t}}{\xi^2}\,d\xi
=&\,\frac14t^{\frac14}\int_0^{\infty}\eta^{-\frac54}(1-e^{-\eta})\,d\eta \\
=&\,t^{\frac14}\int_0^{\infty}\eta^{-\frac14}e^{-\eta}\,d\eta=\Gamma\Bigl(\frac34\Bigr)t^{\frac14}
\end{align*}
since the l'Hospital's rule implies that 
\[
\lim_{\eta\to0^+}\frac{1-e^{-\eta}}{\eta^{\frac14}}=\lim_{\eta\to0^+}\frac{e^{-\eta}}{\,\dfrac14\eta^{-\frac34}\,}
=\lim_{\eta\to0^+}4\eta^{\frac34}e^{-\eta}=0.
\]
Thus it follows that 
\begin{equation}\label{bound_U1}
|U_1(x,t)|\le\frac2{\pi}\,\Gamma\Bigl(\frac34\Bigr)t^{\frac14}.
\end{equation}
Regarding $U_2$, \eqref{int_Kb} and \eqref{ineq_2} with $q_1=0$ yield that 
\begin{align}
|U_2(x,t)|
%\le&\,\frac2{\pi}\int_0^t\int_0^{\infty}|b(\tau)|e^{-\xi^4(t-\tau)}\,d\xi d\tau \notag \\
\le&\,\frac1{2\pi}\,\Gamma\Bigl(\frac14\Bigr)B\Bigl(\frac12,\frac34\Bigr)\|b\|_{B_{\frac12}((0,T];\mathbb{R})}t^{\frac14}. \label{bound_U2}
\end{align}
Hence \eqref{bound_U1} and \eqref{bound_U2} give \eqref{t-0_limit}. 
\end{proof}
\noindent
By virtue of \eqref{U_LP} and Lemma \ref{lem:limit}, we can verify that $U$ satisfies the boundary conditions 
and the initial condition for \eqref{LP}. 

Let us show the estimates of $U$ in 
\[
B_{\frac{2+\gamma}4}((0,T];BUC^{3+\gamma}(\mathbb{R}_+))
\cap C_{\frac{2+\gamma}4}^{\frac{1+\gamma}4}((0,T];BUC^2(\mathbb{R}_+))
\]
with $\gamma\in(0,1)$. Taking account of Remark \ref{rem:reg_wxx2}, we obtain the following lemma. 
Hereafter, we use the symbol $\,\lesssim\,$ if a constant of each inequality depends only on unessential parameters. 
That is, we simply write $a\lesssim b$ instead of $a\le Cb$. 

%------------------------------
\begin{proposition}\label{prop:linear-est}
Let $b\in C_{\frac{3+\gamma}4}^{\frac{1+\gamma}4}((0,T];\mathbb{R})$ with $\gamma\in(0,1)$ and 
let $U$ be the solution to the linear problem \eqref{LP} given by \eqref{U_LP}. Then there exist $C_i>0\,(i=1,2)$ such that 
\begin{align*}
&\|U\|_{B_{\frac{2+\gamma}4}((0,T];BUC^{3+\gamma}(\mathbb{R}_+))}
+\|U\|_{C_{\frac{2+\gamma}4}^{\frac{1+\gamma}4}((0,T];BUC^2(\mathbb{R}_+))} \\
&\le\bigl(C_1+C_2\|b\|_{C_{\frac{3+\gamma}4}^{\frac{1+\gamma}4}((0,T];\mathbb{R})}\bigr)\tan\beta.
\end{align*}
\end{proposition}
%------------------------------

\begin{proof}
\textit{Step 1.} 
Let us derive the estimate of $\|U\|_{B_{\frac{2+\gamma}4}((0,T];BUC^{3+\gamma}(\mathbb{R}_+))}$. 
As for $\|U(t)\|_\infty$, recalling \eqref{bound_U1} and \eqref{bound_U2}, 
%
%%%%%%%%%%%%%%%
\begin{comment}

As for $\|U_1(t)\|_\infty$, \eqref{int_2ndD_K} gives 
\[
|U_1(x,t)|\le\frac2{\pi}\int_0^{\infty}\frac{1-e^{-\xi^4t}}{\xi^2}\,d\xi. 
\]
Applying the change of variable by $\eta=\xi^4t$ and the integration by parts, we have 
\begin{align*}
\int_0^{\infty}\frac{1-e^{-\xi^4t}}{\xi^2}\,d\xi
=&\,\frac14t^{\frac14}\int_0^{\infty}\eta^{-\frac54}(1-e^{-\eta})\,d\eta \\
=&\,t^{\frac14}\int_0^{\infty}\eta^{-\frac14}e^{-\eta}\,d\eta=\Gamma\Bigl(\frac34\Bigr)t^{\frac14}
\end{align*}
since the l'Hospital's rule implies that 
\[
\lim_{\eta\to0^+}\frac{1-e^{-\eta}}{\eta^{\frac14}}=\lim_{\eta\to0^+}\frac{e^{-\eta}}{\,\dfrac14\eta^{-\frac34}\,}
=\lim_{\eta\to0^+}4\eta^{\frac34}e^{-\eta}=0.
\]
Thus it follows that 
\begin{equation}\label{bound_U1}
\|U_1(t)\|_\infty\le\frac2{\pi}\,\Gamma\Bigl(\frac34\Bigr)t^{\frac14}.
\end{equation}
Concerning $\|U_2(t)\|_\infty$, \eqref{int_Kb} and \eqref{ineq_2} with $q_1=0$ yield that 
\begin{align}
\|U_2(t)\|_\infty
%\le&\,\frac2{\pi}\int_0^t\int_0^{\infty}|b(\tau)|e^{-\xi^4(t-\tau)}\,d\xi d\tau \notag \\
\le&\,\frac1{2\pi}\,\Gamma\Bigl(\frac14\Bigr)B\Bigl(\frac12,\frac34\Bigr)\|b\|_{B_{\frac12}((0,T];\mathbb{R})}t^{\frac14}. \label{bound_U2}
\end{align}

\end{comment}
%%%%%%%%%%%%%%%
%
there exist $C_i>0\,(i=1,2)$ such that 
\[
\frac1{T^{\frac{3+\gamma}4}}\sup_{t\in[0,T]}t^{\frac{2+\gamma}4}\|U(t)\|_\infty
\le\bigl(C_1+C_2\|b\|_{B_{\frac12}((0,T];\mathbb{R})}\bigr)\tan\beta.
\]
Let us consider $\|\partial_xU_1(t)\|_\infty$. For any $\rho>0$, we observe 
\begin{align*}
(\partial_xU_1)(y,t)
=&\,\frac1{\rho}\int_y^{y+\rho}\bigl\{(\partial_xU_1)(y,t)-(\partial_xU_1)(x,t)\bigr\}\,dx \\
&\,+\frac1{\rho}\bigl\{U_1(y+\rho,t)-U_1(y,t)\bigr\}. 
\end{align*}
Using \eqref{1stD_U1} and \eqref{ineq_1} with $p=0$ and applying the mean value theorem, 
we see that for $y\le x\le y+\rho$
\begin{align*}
%&\hspace*{-8pt}
|(\partial_xU_1)(x,t)-(\partial_xU_1)(y,t)|
\le&\,\frac2{\pi}\int_0^\infty\frac{e^{-\xi^4t}}{\xi}|\sin(x\xi)-\sin(y\xi)|\,d\xi \\
%=&\,\frac2{\pi}\int_0^\infty\frac{e^{-\xi^4t}}{\xi}\cdot\xi(x-y)|\cos((y+\theta(x-y))\xi)|\,d\xi \\
=&\,\frac2{\pi}(x-y)\int_0^\infty e^{-\xi^4t}|\cos((y+\theta(x-y))\xi)|\,d\xi \\
\le&\,\frac1{2\pi}\,\Gamma\Bigl(\frac14\Bigr)t^{-\frac14}(x-y),
%\le&\,\frac2{\pi}(x-y)\int_0^\infty e^{-\xi^4t}\,d\xi \\
%=&\,\frac1{2\pi}\,\Gamma\Bigl(\frac14\Bigr)t^{-\frac14}(x-y),
\end{align*}
where $\theta\in(0,1)$. 
Integrating on the interval $[y,y+\rho]$, we obtain the inequality 
\[
%\begin{align*}
\frac1{\rho}\int_y^{y+\rho}|(\partial_xU_1)(y,t)-(\partial_xU_1)(x,t)|\,dx
%\le&\,\frac1{2\pi\rho}\Gamma\Bigl(\frac14\Bigr)t^{-\frac14}\int_y^{y+\rho}(x-y)\,dx \\
\le\frac1{4\pi}\,\Gamma\Bigl(\frac14\Bigr)t^{-\frac14}\rho.
%\end{align*}
\]
On the other hand, it follows from \eqref{bound_U1} that  
\begin{align*}
\frac1{\rho}|U_1(y+\rho,t)-U_1(y,t)|\le\frac2{\rho}\|U_1(t)\|_\infty\le\frac4{\pi}\,\Gamma\Bigl(\frac34\Bigr)\frac{t^{\frac14}}{\rho}. 
\end{align*}
As a result, for any $\rho>0$
\[
\|\partial_xU_1(t)\|_\infty
\le\frac1{4\pi}\,\Gamma\Bigl(\frac14\Bigr)t^{-\frac14}\rho+\frac4{\pi}\,\Gamma\Bigl(\frac34\Bigr)\frac{t^{\frac14}}{\rho}.
\]
Set 
\[
h(\rho):=C_{\ast,1}t^{-\frac14}\rho+\frac{C_{\ast,2}t^{\frac14}}{\rho}, \quad 
C_{\ast,1}=\frac1{4\pi}\,\Gamma\Bigl(\frac14\Bigr),\,\ C_{\ast,2}=\frac4{\pi}\,\Gamma\Bigl(\frac34\Bigr).
\]
$h(\rho)$ is minimal at 
$\rho=(C_{\ast,2}t^{\frac14}/C_{\ast,1}t^{-\frac14})^{\frac12}=(C_{\ast,2}/C_{\ast,1})^{\frac12}t^{\frac14}\,(=:\rho_\ast)$ 
and its value is 
\[
h(\rho_\ast)=2(C_{\ast,1}C_{\ast,2})^{\frac12}. 
\]
Thus we have 
\begin{equation}\label{bound_1stD_U1}
\|\partial_xU_1(t)\|_\infty\le h(\rho_\ast)=2(C_{\ast,1}C_{\ast,2})^{\frac12}.
\end{equation}
Let us show the boundedness of  $\|\partial_x^kU_1(t)\|_\infty\,(k=2,3)$ and $\|\partial_x^\ell U_2(t)\|_\infty\,(\ell=1,2)$. 
Recalling \eqref{2ndD_U1}, \eqref{3rdD_U1}, and \eqref{kthD_U2_1} and applying \eqref{ineq_1} with $p=0,1$ 
and \eqref{ineq_2} with $q_1=\ell$, we see that for $k=2,3$ and $\ell=1,2$
%\begin{align*}
%\label{kthD_U1}
%&|\partial_x^kU_1(x,t)|\le\frac2{\pi}\int_0^{\infty}\xi^{k-2}e^{-\xi^4t}\,d\xi, \\
%\label{lthD_U2}
%&|\partial_x^\ell U_2(x,t)|\le\frac2{\pi}\int_0^t\int_0^{\infty}|b(\tau)|\,\xi^\ell e^{-\xi^4(t-\tau)}\,d\xi d\tau. 
%\end{align*}
\begin{align}
%\label{basic-ineq-1}
\|\partial_x^kU_1(t)\|_\infty
%\le&\,\frac14\,t^{(1-k)/4}\int_0^{\infty}\xi^{(k-1)/4-1}e^{-\xi}\,d\xi \\
\le&\,\frac1{2\pi}\,\Gamma\Bigl(\frac{k-1}4\Bigr)t^{\frac{1-k}4},  \label{bound_kthD_U1} \\
%\label{basic-ineq-2}
\|\partial_x^\ell U_2(t)\|_\infty
%\le&\,\frac14\int_0^{\infty}\int_0^t|b(\tau)|(t-\tau)^{-(\ell+1)/4}\xi^{(\ell-3)/4}e^{-\xi}\,d\tau d\xi \\
%\le&\,\frac14\Gamma\Biglt(\frac{\ell+1}4\Bigr)\|b\|_{B_{\frac12}((0,T];\mathbb{R})}\int_0^t\tau^{-\frac12}(t-\tau)^{-(\ell+1)/4}\,d\tau \notag \\
\le&\,\frac1{2\pi}\,\Gamma\Bigl(\frac{\ell+1}4\Bigr)B\Bigl(\frac12,\frac{3-\ell}4\Bigr)\|b\|_{B_{\frac12}((0,T];\mathbb{R})}t^{\frac{1-\ell}4}. 
\label{bound_lthD_U2}
\end{align}
\eqref{bound_1stD_U1}, \eqref{bound_kthD_U1} with $k=2$ and \eqref{bound_lthD_U2} yield that there exist $C_i > 0\,(i=1,2)$ such that
\[
\sum_{k=1}^2\frac{1}{T^{\frac{3-k+\gamma}4}}\sup _{t\in(0,T]}t^{\frac{2+\gamma}4}\|\partial_x^kU(t)\|_{\infty}
\le\bigl(C_1+C_2\|b\|_{B_{\frac12}((0,T];\mathbb{R})}\bigr)\tan\beta.
\]
%\[
%\frac1{T^{\frac{2+\gamma}4}}\sup_{t\in[0,T]}t^{\frac{2+\gamma}4}\|\partial_xU(t)\|_\infty
%+\frac1{T^{\frac{1+\gamma}4}}\sup_{t\in[0,T]}t^{\frac{2+\gamma}4}\|\partial_x^2U(t)\|_\infty
%\le\bigl(C_1+C_2\|b\|_{B_{\frac12}((0,T];\mathbb{R})}\bigr)\tan\beta.
%\]
Let us consider $\|\partial_x^3U_2(t)\|_\infty$. Recalling \eqref{3rdD_U2} and using \eqref{bound_1stD_U1}, 
the first term of $\|\partial_x^3U_2(t)\|_\infty$ is estimated by 
\begin{align*}
(\text{The first term of $\|\partial_x^3U_2(t)\|_\infty$})
\le&\,|b(t)|\|\partial_xU_1(t)\|_\infty \\
\le&\,2(C_{\ast,1}C_{\ast,2})^{\frac12}\|b\|_{B_{\frac12}((0,T];\mathbb{R})}t^{-\frac12}.
\end{align*}
Furthermore, it follows from \eqref{ineq_3} with $q_2=3$ that 
\begin{align*}
%&\hspace*{-8pt}
(\text{The second term of $\|\partial_x^3U_2(t)\|_\infty$}) %\\
%\le&\,\frac2{\pi}\int_0^t\int_0^{\infty}|b(\tau)-b(t)|\,\xi^3e^{-\xi^4(t-\tau)}\,d\xi d\tau \\
%=&\,\frac1{2\pi}\int_0^t|b(\tau)-b(t)|\,(t-\tau)^{-1}\,d\tau \\
%\le&\,\frac1{2\pi}\|b\|_{C_{\frac{3+\gamma}4}^{\frac{1+\gamma}4}((0,T];\mathbb{R})}
%\int_0^t(t-\tau)^{\frac{1+\gamma}4-1}\,\tau^{-\frac{3+\gamma}4}\,d\tau \\
\le&\,\frac1{2\pi}B\Bigl(\frac{1-\gamma}4,\frac{1+\gamma}4\Bigr)
\|b\|_{C_{\frac{3+\gamma}4}^{\frac{1+\gamma}4}((0,T];\mathbb{R})}t^{-\frac12}.
\end{align*}
These inequalities and \eqref{bound_kthD_U1} with $k=3$ imply that there exist $C_i > 0\,(i=1,2)$ such that 
\[
\frac1{T^{\frac{\gamma}4}}\sup_{t\in(0,T]}t^{\frac{2+\gamma}4}\|\partial_x^3U(t)\|_\infty
\le\bigl(C_1+C_2\|b\|_{B_{\frac12}((0,T];\mathbb{R})}\bigr)\tan\beta.
\]
Finally, let us derive the estimates of $[\partial_x^3U_1(t)]_\gamma$ and $[\partial_x^3U_2(t)]_\gamma$. 
Since $[\sin(\,\cdot\,\xi)]_\gamma\le c_0\xi^{\gamma}$ is obtained, we have
\begin{align*}
[\partial_x^3U_1(t)]_\gamma
\le&\,\frac{2c_0}{\pi}\int_0^{\infty}\xi^{1+\gamma}e^{-\xi^4t}\,d\xi, \\
[\partial_x^3U_2(t)]_\gamma
\le&\,\frac{2c_0}{\pi}|b(t)|\int_0^{\infty}\xi^{-1+\gamma}e^{-\xi^4t}\,d\xi \\
&\,+\frac{2c_0}{\pi}\int_0^{\infty}\int_0^t|b(\tau)-b(t)|\,\xi^{3+\gamma}e^{-\xi^4(t-\tau)}\,d\tau d\xi \\
=:&\,I_1+I_2. 
\end{align*}
To derive these inequalities, we have used \eqref{3rdD_U1} and \eqref{3rdD_U2} with \eqref{1stD_U1}. 
Regarding $[\partial_x^3U_1(t)]_\gamma$, \eqref{ineq_1} with $p=1+\gamma$ implies that 
\[
[\partial_x^3U_1(t)]_\gamma\le\frac{c_0}{2\pi}\,\Gamma\Bigl(\frac{2+\gamma}4\Bigr)t^{-\frac{2+\gamma}4}.
\]
%so that we obtain 
%\[
%\sup_{t\in(0,T]}t^{\frac{2+\gamma}4}[\partial_x^3U_1(t)]_\gamma\le\frac{c_0}{2\pi}\,\Gamma\Bigl(\frac{2+\gamma}4\Bigr).
%\]
As for $[\partial_x^3U_2(t)]_\gamma$, it follows from \eqref{ineq_2} with $q_1=-1+\gamma$ 
and \eqref{ineq_3} with $q_2=3+\gamma$ that
\begin{align*}
I_1\lesssim&\,\Gamma\Bigl(\frac{\gamma}4\Bigr)\|b\|_{B_{\frac12}((0,T];\mathbb{R})}t^{-\frac{2+\gamma}4}, \\
I_2\lesssim&\,\Gamma\Bigl(1+\frac{\gamma}4\Bigr)B\Bigl(\frac{1-\gamma}4,\frac14\Bigr)
\|b\|_{C_{\frac{3+\gamma}4}^{\frac{1+\gamma}4}((0,T];\mathbb{R})}t^{-\frac{2+\gamma}4}.
\end{align*}
Thus we have
\[
[\partial_x^3U_2(t)]_\gamma
\lesssim\biggl\{\Gamma\Bigl(\frac{\gamma}4\Bigr)+\Gamma\Bigl(1+\frac{\gamma}4\Bigr)B\Bigl(\frac{1-\gamma}4,\frac14\Bigr)\biggr\}
\|b\|_{C_{\frac{3+\gamma}4}^{\frac{1+\gamma}4}((0,T];\mathbb{R})}t^{-\frac{2+\gamma}4}.
\]
As a result, there exist $C_i > 0\,(i=1,2)$ such that 
\[
\sup_{t\in(0,T]}t^{\frac{2+\gamma}4}[\partial_x^3U(t)]_\gamma
\le\bigl(C_1+C_2\|b\|_{C_{\frac{3+\gamma}4}^{\frac{1+\gamma}4}((0,T];\mathbb{R})}\bigr)\tan\beta.
\]

\textit{Step 2.} 
Let us derive the estimate of $\|U\|_{C_{\frac{2+\gamma}4}^{\frac{1+\gamma}4}((0,T];BUC^2(\mathbb{R}_+))}$. 
As a first step, we consider the estimate of $\partial_x^kU_1\,(k=0,1,2)$. 
%Note that 
%\[
%\partial_x^k\cos(x\xi)=\xi^k\cos\Bigl(x\xi+\frac{k\pi}2\Bigr).
%\]
Then it follows that for $0<s<t\le T$
\begin{align*}
&\hspace*{-8pt}
\partial_x^kU_1(x,t)-\partial_x^kU_1(x,s) \\
=&\,-\frac2{\pi}\int_0^{\infty}\biggl(\frac{1-e^{-\xi^4t}}{\xi^2}-\frac{1-e^{-\xi^4s}}{\xi^2}\biggr)\xi^k\cos\Bigl(x\xi+\frac{k\pi}2\Bigr)\,d\xi \\
%=&\,\frac2{\pi}\int_0^{\infty}\xi^{k-2}(e^{-\xi^4t}-e^{-\xi^4s})\cos\Bigl(x\xi+\frac{k\pi}2\Bigr)\,d\xi \\
%=&\,\frac2{\pi}\int_0^{\infty}\xi^{k-2}\biggl(\int_s^t\partial_\sigma(e^{-\xi^4\sigma})\,d\sigma\biggr)\cos\Bigl(x\xi+\frac{k\pi}2\Bigr)\,d\xi \\
=&\,-\frac2{\pi}\int_s^t\int_0^{\infty}\xi^{k+2}e^{-\xi^4\sigma}\cos\Bigl(x\xi+\frac{k\pi}2\Bigr)\,d\xi d\sigma.
\end{align*}
This and \eqref{ineq_1} with $p=k+2$ imply that 
\begin{align*}
&\hspace*{-8pt}
|\partial_x^kU_1(x,t)-\partial_x^kU_1(x,s)| \\
%\le&\,\frac2{\pi}\int_s^t\int_0^{\infty}\xi^{k+2}e^{-\xi^4\sigma}\,d\xi d\sigma
\le&\,\frac1{2\pi}\Gamma\Bigl(\frac{k+3}4\Bigr)\int_s^t\sigma^{-\frac{k+3}4}\,d\sigma
\lesssim s^{-\frac{2+\gamma}4}\int_s^t\sigma^{-\frac{k+1}4+\frac{\gamma}4}\,d\sigma \\
\lesssim&\,s^{-\frac{2+\gamma}4}(t^{\frac{3-k+\gamma}4}-s^{\frac{3-k+\gamma}4})
%\le&\,\frac1{2\pi}\Gamma\Bigl(\frac{k+3}4\Bigr)s^{-\frac{2+\gamma}4}\int_s^t\sigma^{-\frac{k+1}4+\frac{\gamma}4}\,d\sigma \\
%\le&\,\frac2{\pi(3-k+\gamma)}\Gamma\Bigl(\frac{k+3}4\Bigr)s^{-\frac{2+\gamma}4}(t^{\frac{3-k+\gamma}4}-s^{\frac{3-k+\gamma}4}) \\
%\le&\,\frac2{\pi(3-k+\gamma)}\Gamma\Bigl(\frac{k+3}4\Bigr)s^{-\frac{2+\gamma}4}(t-s)^{\frac{3-k+\gamma}4} \\
\lesssim s^{-\frac{2+\gamma}4}(t-s)^{\frac{1+\gamma}4}T^{\frac{2-k}4}.
%\le&\,\frac2{\pi(3-k+\gamma)}\Gamma\Bigl(\frac{k+3}4\Bigr)s^{-\frac{2+\gamma}4}(t-s)^{\frac{1+\gamma}4}T^{\frac{2-k}4}.
\end{align*}
Thus we see that there exists $C_1>0$ such that for $0<s<t\le T$
\[
\sum_{k=0}^2\frac1{T^{\frac{2-k}4}}\sup _{\begin{subarray}{c}s,t\in(0,T] \\ s<t\end{subarray}}s^{\frac{2+\gamma}4}
\frac{\|\partial_x^kU_1(\,\cdot\,,t)-\partial_x^kU_1(\,\cdot\,,s)\|_\infty}{(t-s)^{\frac{1+\gamma}4}}\le C_1.
%\le\frac2{\pi(3-k+\gamma)}\Gamma\Bigl(\frac{k+3}4\Bigr).
\]
Let us consider the estimate of $\partial_x^kU_2\,(k=0,1,2)$. Using \eqref{kthD_U2_2}, 
we observe that $\partial_x^kU_2(x,t)-\partial_x^kU_2(x,s)$ is transformed into 
\[
\partial_x^kU_2(x,t)-\partial_x^kU_2(x,s)=:\frac2{\pi}\sum_{i=1}^5J_i
\]
for $0<s<t\le T$, where
\begin{align*}
J_1:=&\,b(t)\int_0^{\infty}\frac{e^{-\xi^4s}-e^{-\xi^4t}}{\xi^4}\,\xi^k\cos\Bigl(x\xi+\frac{k\pi}2\Bigr)\,d\xi, \\
J_2:=&\,(b(t)-b(s))\int_0^{\infty}\frac{1-e^{-\xi^4s}}{\xi^4}\,\xi^k\cos\Bigl(x\xi+\frac{k\pi}2\Bigr)\,d\xi, \\
J_3:=&\,\int_s^t\int_0^\infty(b(t)-b(\tau))\xi^ke^{-\xi^4(t-\tau)}\cos\Bigl(x\xi+\frac{k\pi}2\Bigr)\,d\xi d\tau, \\
J_4:=&\,(b(t)-b(s))\int_0^s\int_0^\infty\xi^ke^{-\xi^4(t-\tau)}\cos\Bigl(x\xi+\frac{k\pi}2\Bigr)\,d\xi d\tau, \\
J_5:=&\,\int_0^s\int_0^\infty(b(s)-b(\tau))\xi^k(e^{-\xi^4(t-\tau)}-e^{-\xi^4(s-\tau)})\cos\Bigl(x\xi+\frac{k\pi}2\Bigr)\,d\xi d\tau.
\end{align*}
Let us first derive the estimate of $J_1$. By virtue of \eqref{ineq_1} with $p=k$ and $b\in B_{\frac12}((0,T];\mathbb{R})$, 
we have 
\begin{align*}
|J_1|\le&\,|b(t)|\int_0^{\infty}e^{-\xi^4s}\,\frac{1-e^{-\xi^4(t-s)}}{\xi^4}\,\xi^k\,d\xi \\
\le&\,|b(t)|\int_0^{t-s}\int_0^{\infty}\xi^ke^{-\xi^4\sigma}\,d\xi d\sigma \\
\le&\,\frac14\,\Gamma\Bigl(\frac{k+1}4\Bigr)\|b\|_{B_{\frac12}((0,T];\mathbb{R})}
t^{-\frac12}\int_0^{t-s}\sigma^{-\frac{k+1}4}\,d\sigma \\
\lesssim&\,\|b\|_{B_{\frac12}((0,T];\mathbb{R})}s^{-\frac{2+\gamma}4}s^{\frac{\gamma}4}(t-s)^{\frac{3-k}4} \\
\lesssim&\,\|b\|_{C_{\frac{3+\gamma}4}^{\frac{1+\gamma}4}((0,T];\mathbb{R})}
s^{-\frac{2+\gamma}4}(t-s)^{\frac{1+\gamma}4}T^{\frac{2-k}4}.
\end{align*}
In the last inequality, we have used 
\[
(t-s)^{\frac{3-k}4}\le(t-s)^{\frac{1+\gamma}4+\frac{2-k-\gamma}4}
\le(t-s)^{\frac{1+\gamma}4}T^{\frac{2-k-\gamma}4}.
\] 
As for $J_2$, \eqref{ineq_1} with $p=k$ and $b\in C_{\frac{3+\gamma}4}^{\frac{1+\gamma}4}((0,T];\mathbb{R})$ imply that 
\begin{align*}
|J_2|\le&\,|b(t)-b(s)|\int_0^{\infty}\frac{1-e^{-\xi^4s}}{\xi^4}\,\xi^k\,d\xi \\
\le&\,|b(t)-b(s)|\int_0^s\int_0^\infty\xi^ke^{-\xi^4\sigma}\,d\xi d\sigma \\
\le&\,\frac14\,\Gamma\Bigl(\frac{k+1}4\Bigr)\|b\|_{C_{\frac{3+\gamma}4}^{\frac{1+\gamma}4}((0,T];\mathbb{R})}
s^{-\frac{3+\gamma}4}(t-s)^{\frac{1+\gamma}4}\int_0^s\sigma^{-\frac{k+1}4}\,d\sigma \\
\lesssim&\,\|b\|_{C_{\frac{3+\gamma}4}^{\frac{1+\gamma}4}((0,T];\mathbb{R})}
s^{-\frac{2+\gamma}4}(t-s)^{\frac{1+\gamma}4}T^{\frac{2-k}4}.
\end{align*}
Concerning $J_3$, an argument of the proof of \eqref{ineq_3} yields 
\begin{align*}
|J_3|\le&\,\int_s^t\int_0^\infty|b(t)-b(\tau)|\xi^ke^{-\xi^4(t-\tau)}\,d\xi d\tau \\
\le&\,\frac14\,\Gamma\Bigl(\frac{k+1}4\Bigr)\|b\|_{C_{\frac{3+\gamma}4}^{\frac{1+\gamma}4}((0,T];\mathbb{R})}
\int_s^t(t-\tau)^{\frac{\gamma-k}4}\tau^{-\frac{3+\gamma}4}\,d\tau \\
\lesssim&\,\|b\|_{C_{\frac{3+\gamma}4}^{\frac{1+\gamma}4}((0,T];\mathbb{R})}s^{-\frac{2+\gamma}4}(t-s)^{\frac{3-k+\gamma}4}
\int_s^t(t-\tau)^{-\frac34}\tau^{-\frac14}\,d\tau \\
\lesssim&\,B\Bigl(\frac14,\frac34\Bigr)\|b\|_{C_{\frac{3+\gamma}4}^{\frac{1+\gamma}4}((0,T];\mathbb{R})}
s^{-\frac{2+\gamma}4}(t-s)^{\frac{3-k+\gamma}4} \\
\lesssim&\,\|b\|_{C_{\frac{3+\gamma}4}^{\frac{1+\gamma}4}((0,T];\mathbb{R})}
s^{-\frac{2+\gamma}4}(t-s)^{\frac{1+\gamma}4}T^{\frac{2-k}4}.
\end{align*}
As regards $J_4$, it follows form \eqref{ineq_1} with $p=k$ and $\sigma=t-\tau$ that 
\begin{align*}
|J_4|\le&\,|b(t)-b(s)|\int_0^s\int_0^\infty\xi^ke^{-\xi^4(t-\tau)}\,d\xi d\tau \\
\le&\,\frac14\,\Gamma\Bigl(\frac{k+1}4\Bigr)\|b\|_{C_{\frac{3+\gamma}4}^{\frac{1+\gamma}4}((0,T];\mathbb{R})}
s^{-\frac{3+\gamma}4}(t-s)^{\frac{1+\gamma}4}\int_0^s(t-\tau)^{-\frac{k+1}4}\,d\tau \\
\lesssim&\,\|b\|_{C_{\frac{3+\gamma}4}^{\frac{1+\gamma}4}((0,T];\mathbb{R})}
s^{-\frac{3+\gamma}4}(t-s)^{\frac{1+\gamma}4}\int_0^s(s-\tau)^{-\frac{k+1}4}\,d\tau \\
\lesssim&\,\|b\|_{C_{\frac{3+\gamma}4}^{\frac{1+\gamma}4}((0,T];\mathbb{R})}
s^{-\frac{3+\gamma}4}(t-s)^{\frac{1+\gamma}4}s^{\frac{3-k}4} \\
\lesssim&\,\|b\|_{C_{\frac{3+\gamma}4}^{\frac{1+\gamma}4}((0,T];\mathbb{R})}
s^{-\frac{2+\gamma}4}(t-s)^{\frac{1+\gamma}4}T^{\frac{2-k}4}.
\end{align*}
Finally, using \eqref{ineq_1} with $p=k+4$ and $b\in C_{\frac{3+\gamma}4}^{\frac{1+\gamma}4}((0,T];\mathbb{R})$, 
the upper bound of $J_5$ is given by 
\begin{align*}
|J_5|\le&\,\int_0^s\int_0^\infty|b(s)-b(\tau)|\xi^k(e^{-\xi^4(s-\tau)}-e^{-\xi^4(t-\tau)})\,d\xi d\tau \\
=&\,\int_0^s|b(s)-b(\tau)|\int_{s-\tau}^{t-\tau}\biggl(\int_0^\infty\xi^{k+4}e^{-\xi^4\sigma}\,d\xi\biggr)d\sigma d\tau \\
\le&\,\frac14\,\Gamma\Bigl(\frac{k+5}4\Bigr)\|b\|_{C_{\frac{3+\gamma}4}^{\frac{1+\gamma}4}((0,T];\mathbb{R})}
\int_0^s\tau^{-\frac{3+\gamma}4}(s-\tau)^{\frac{1+\gamma}4}
\biggl(\int_{s-\tau}^{t-\tau}\sigma^{-\frac{k+5}4}\,d\sigma\biggr)d\tau \\
\lesssim&\,\|b\|_{C_{\frac{3+\gamma}4}^{\frac{1+\gamma}4}((0,T];\mathbb{R})}
\int_0^s\tau^{-\frac{3+\gamma}4}(s-\tau)^{-\frac{1+k}4}
\biggl(\int_{s-\tau}^{t-\tau}\sigma^{\frac{1+\gamma}4-1}\,d\sigma\biggr)d\tau \\
\lesssim&\,B\Bigl(\frac{1-\gamma}4,\frac{3-k}4\Bigr)\|b\|_{C_{\frac{3+\gamma}4}^{\frac{1+\gamma}4}((0,T];\mathbb{R})}
s^{-\frac{k+\gamma}4}(t-s)^{\frac{1+\gamma}4} \\
\lesssim&\,\|b\|_{C_{\frac{3+\gamma}4}^{\frac{1+\gamma}4}((0,T];\mathbb{R})}
s^{-\frac{2+\gamma}4}(t-s)^{\frac{1+\gamma}4}T^{\frac{2-k}4}.
\end{align*}
As a result, there exists $C_2>0$ such that for $0<s<t\le T$
\[
\sum_{k=0}^2\frac1{T^{\frac{2-k}4}}\sup _{\begin{subarray}{c}s,t\in(0,T] \\ s<t\end{subarray}}s^{\frac{2+\gamma}4}
\frac{\|\partial_x^kU_2(\,\cdot\,,t)-\partial_x^kU_2(\,\cdot\,,s)\|_\infty}{(t-s)^{\frac{1+\gamma}4}}
\le C_2\|b\|_{C_{\frac{3+\gamma}4}^{\frac{1+\gamma}4}((0,T];\mathbb{R})}.
\]
This completes the proof. 
\end{proof}

%------------------------------
\begin{corollary}\label{cor:differ-linear-est}
Let $b^{(j)}\in C_{\frac{3+\gamma}4}^{\frac{1+\gamma}4}((0,T];\mathbb{R})\,(j=1,2)$ with $\gamma\in(0,1)$ and 
let $U^{(j)}$ be the solution to the linear problem \eqref{LP} with $b^{(j)}$ instead of $b$. 
Then there exists $C>0$ such that 
\begin{align*}
&\|U^{(1)}-U^{(2)}\|_{B_{\frac{2+\gamma}4}((0,T];BUC^{3+\gamma}(\mathbb{R}_+))}
+\|U^{(1)}-U^{(2)}\|_{C_{\frac{2+\gamma}4}^{\frac{1+\gamma}4}((0,T];BUC^2(\mathbb{R}_+))} \\
&\le C\|b^{(1)}-b^{(2)}\|_{C_{\frac{3+\gamma}4}^{\frac{1+\gamma}4}((0,T];\mathbb{R})}\tan\beta.
\end{align*}
\end{corollary}
%------------------------------

\begin{proof}
From the linearity and the fact that $U_1$ does not depend on $b^{(j)}$, the statement follows. 
%Since the first term in \eqref{U_LP} does not depend on $b$, we have
%\begin{align*}
%&\hspace*{-8pt}
%U^{(1)}(x,t)-U^{(2)}(x,t) \\
%=&\,c_\beta\int_0^tK(x,0,t-\tau)\{b^{(1)}(\tau)-b^{(2)}(\tau)\}\,d\tau \\
%=&\,\frac{2c_\beta}{\pi}\int_0^t\int_0^{\infty}\{b^{(1)}(\tau)-b^{(2)}(\tau)\}e^{-\xi^4(t-\tau)}\cos(x\xi)\,d\xi d\tau.
%\end{align*}
%Applying an argument similar to the proof of Proposition \ref{prop:linear-est}, we can obtain the desired result. 
\end{proof}

%=======================================
\section{Existence of a self-similar solution to the Mullins problem}\label{sec:MullinsProblem}
%=======================================
%=======================================
\subsection{The estimates of the Green function}\label{subsec:Grenn-func}
%=======================================
Recalling the first equality of \eqref{GreenFunc} and applying the change of variable $\xi=\eta/t^{\frac14}$, 
the derivatives of the Green function $G(x,t)$ are represented as 
\begin{align}
\partial_t^\ell\partial_x^kG(x,t)
%=&\,\frac1{2\pi}\int_{\mathbb{R}}(-\xi^4)^\ell(\textbf{i}\xi)^ke^{-\xi^4t+\textbf{i}x\xi}\,d\xi \\
=&\,\frac{(-1)^\ell\,\textbf{i}^k}{2\pi}\int_{\mathbb{R}}\xi^{4\ell+k}e^{-\xi^4t+\textbf{i}x\xi}\,d\xi \notag \\
=&\,\frac{(-1)^\ell\,\textbf{i}^k}{2\pi t^{\frac{4\ell+k+1}4}}\int_{\mathbb{R}}\eta^{4\ell+k}e^{-\eta^4+\textbf{i}(x/t^{\frac14})\eta}\,d\eta \label{DGreenFunc-1}
\end{align}
for $\ell,k\in\mathbb{N}\cup\{0\}$. 

%------------------------------
\begin{remark}\label{rem:even-odd}
Since $\partial_x^k\cos(x\xi)=\xi^k\cos(x\xi+(k\pi)/2)$ for $k\in\mathbb{N}\cup\{0\}$, the second equality of \eqref{GreenFunc} implies that 
\begin{align}
\partial_t^\ell\partial_x^kG(x,t)
%=&\,\frac1{\pi}\int_0^\infty(-\xi^4)^\ell\xi^ke^{-\xi^4t}\cos\Bigl(x\xi+\frac{k\pi}2\Bigr)\,d\xi \\
=&\,\frac{(-1)^\ell}{\pi}\int_0^\infty\xi^{4\ell+k}e^{-\xi^4t}\cos\Bigl(x\xi+\frac{k\pi}2\Bigr)\,d\xi \notag \\
=&\,\left\{\begin{array}{ll} 
\displaystyle\dfrac{(-1)^{\ell+m}}{\pi}\int_0^\infty\xi^{4\ell+k}e^{-\xi^4t}\cos(x\xi)\,d\xi&(k=2m) \\[0.4cm]
\displaystyle\dfrac{(-1)^{\ell+m+1}}{\pi}\int_0^\infty\xi^{4\ell+k}e^{-\xi^4t}\sin(x\xi)\,d\xi&(k=2m+1)
\end{array}\right. \label{DGreenFunc-2}
\end{align}
for $m\in\mathbb{N}\cup\{0\}$. From \eqref{DGreenFunc-2} we see that $\partial_t^\ell\partial_x^kG(x,t)$ is an even function 
for $k=2m$ and an odd function for $k=2m+1$ with respect to $x$. 
\end{remark}
%------------------------------

%------------------------------
\begin{lemma}\label{lem:int_DG}
Assume that $F\in BUC(\mathbb{R})$. 
Then there exists $C>0$ such that for $\ell,k\in\mathbb{N}\cup\{0\}$
\[
\int_{\mathbb{R}}|\partial_t^\ell\partial_x^kG(x-y,t)||F(y)|\,dy\le Ct^{-\frac{4\ell+k}4}\|F\|_\infty.
\]
\end{lemma}
%------------------------------

\begin{proof}
Since $|\partial_t^\ell\partial_z^kG(z,t)|$ is an even function with respect to $z$ by Remark \ref{rem:even-odd}, we have
\begin{align*}
%\label{est-intGF}
&\hspace*{-8pt}
\int_{\mathbb{R}}|\partial_t^\ell\partial_x^kG(x-y,t)||F(y)|\,dy \\
=&\,\int_{\mathbb{R}}|\partial_t^\ell\partial_z^kG(z,t)||F(x-z)|\,dz \\
%=&\,\int_0^\infty|\partial_t^\ell\partial_z^kG(z,t)||F(x-z)|\,dz+\int_{-\infty}^0|\partial_t^\ell\partial_z^kG(z,t)||F(x-z)|\,dz \\
%=&\,\int_0^\infty|\partial_t^\ell\partial_z^kG(z,t)||F(x-z)|\,dz+\int_{-\infty}^0|\partial_t^\ell\partial_z^kG(-z,t)||F(x-z)|\,dz \\
%=&\,\int_0^\infty|\partial_t^\ell\partial_z^kG(z,t)||F(x-z)|\,dz+\int_0^\infty|\partial_t^\ell\partial_z^kG(z,t)||F(x+z)|\,dz \\
=&\,\int_0^\infty|\partial_t^\ell\partial_z^kG(z,t)|\{|F(x-z)|+|F(x+z)|\}\,dz \notag \\
\le&\,2\int_0^\infty|\partial_t^\ell\partial_z^kG(z,t)|\,dz\,\|F\|_\infty. \notag
\end{align*}
By virtue of Lemma \ref{lem:Cui} and the fact that 
\begin{align*}
\int_0^{\infty}\Bigl(\frac{z}{t^{\frac14}}\Bigr)^{\frac{j}3}e^{-\nu(z/t^{\frac14})^{\frac43}}\,dz
=\frac34\Bigl(\frac1{\nu}\Bigr)^{\frac{j+3}4}\Gamma\Bigl(\frac{j+3}4\Bigr)t^{\frac14} \quad (j=0,4\ell+k) 
\end{align*}
for a constant $\nu>0$, we arrive at the desired result. 
\end{proof}

\subsection{Existence of a bounded self-similar solution}\label{subsec:mild-sol}
%=======================================
In this subsection, we solve the initial boundary value problem \eqref{SSP} to obtain a bounded self-similar solution 
to the Mullins problem. 
%the Mullins problem \eqref{eq:161031a2}--\eqref{eq:161031a5}. 
%Recalling \eqref{transformed_eq} and \eqref{transformed_2ndbc}, the Mullins problem is transformed into 
%\[
%\left\{\begin{array}{l}
%w_t=-w_{xxxx}-\partial_x\Xi(w_x,w_{xx},w_{xxx}),\,\,\ x>0,\ t>0, \\[0.05cm]
%w_x(0,t)=\tan \beta ,\,\,\ t>0, \\[0.05cm]
%w_{xxx}(0,t)=c_{\beta}(w_{xx}(0,t))^2, \,\,\ t>0, \\[0.05cm]
%w(x,0)=a(x),\,\,\ x\ge0.
%\end{array}\right.
%\]

First, we derive the integral equation for $w$ corresponding to the problem \eqref{SSP}. 
Let $w(x,t)$ satisfy \eqref{SSP}. We decompose $w(x,t)$ into $U^w(x,t)$ and $u^w(x,t)$, i.e., $w(x,t)=u^w(x,t)+U^w(x,t)$, 
where $u^w$ satisfies 
\begin{equation}\label{IBVP_u}
\left\{\begin{array}{l}
u_t=-u_{xxxx}-\partial_x\Xi(w_x,w_{xx},w_{xxx}),\,\,\ x>0,\ t>0, \\
u_x(0,t)=u_{xxx}(0,t)=0,\,\,\ t>0,\\
u(x,0)=0,\,\,\ x\in\mathbb{R}_+.
\end{array}\right.
\end{equation}
and $U^w$ fulfills 
\begin{equation*}%\label{LPw}
\left\{\begin{array}{l}
U_t=-U_{xxxx},\,\,\ x>0,\ t>0, \\[0.05cm]
U_x(0,t)=\tan \beta ,\,\,\ t>0, \\[0.05cm]
U_{xxx}(0,t)=c_{\beta}(w_{xx}(0,t))^2, \,\,\ t>0, \\[0.05cm]
U(x,0)=0,\,\,\ x\in\mathbb{R}_+. 
\end{array}\right.
\end{equation*}
Note that $U^w$ is represented as \eqref{U_LP}, where $U_1$ and $U_2$ are given by \eqref{int_2ndD_K} and \eqref{int_Kb} 
with $b(t)=c_{\beta}(w_{xx}(0,t))^2$, respectively. Let us derive the expression for $u^w$. Define an odd extension operator 
$\mathcal{P}:BUC^{\gamma}(\mathbb{R}_+)\to BUC^{\gamma}_{\text{odd}}(\mathbb{R})$ as
\[
%\begin{equation}\label{def_P}
(\mathcal{P}\varphi)(x):=\left\{\begin{array}{ll} \varphi (x)-\varphi(0)&\text{if $x\ge 0$}, \\ 
-(\varphi(-x)-\varphi(0))&\text{if $x<0$}, \end{array}\right.
%\end{equation}
\]
where $BUC^{\gamma}_{\text{odd}}(\mathbb{R}):=\{\varphi\in BUC^{\gamma}(\mathbb{R})\,|\,\varphi(-x)=-\varphi(x)\ (x\in\mathbb{R})\}$. 
Note that if $\varphi\in BUC^{\gamma}(\mathbb{R}_+)$, 
%$\partial_x\Xi=\partial_x(\mathcal{P}\Xi)$ and $\partial_x(\mathcal{P}\Xi)$ is an even function on $\mathbb{R}$. 
%Furthermore, we see that 
\begin{equation}\label{Pf-Holder}
[\mathcal{P}\varphi]_\gamma\le C[\varphi]_\gamma
\end{equation}
holds for a constant $C>0$. Set 
\[
f^w(x,t):=\Xi(w_x,w_{xx},w_{xxx})(x,t),
\]
where the explicit form of $\Xi$ is given by \eqref{def_Xi}. 
Let $\widetilde{u}^w(x,t)$ be the solution to the Cauchy problem 
\begin{equation}\label{Cauchy}
\left\{\begin{array}{l}
\widetilde{u}_t=-\widetilde{u}_{xxxx}-\partial_x(\mathcal{P}f^w)(x,t),\,\,\ x\in\mathbb{R},\ t>0, \\[0.05cm]
\widetilde{u}(x,0)=0,\,\,\ x\in\mathbb{R}. 
\end{array}\right.
\end{equation}
Setting $\widetilde{u}^w_-(x,t):=\widetilde{u}^w(-x,t)$, $\widetilde{u}^w_-$ also satisfies \eqref{Cauchy} since $\partial_x(\mathcal{P}f^w)(x,t)$ 
is an even function. From the uniqueness of a solution to the Cauchy problem \eqref{Cauchy}, we see that 
$\widetilde{u}^w_-(x,t)=\widetilde{u}^w(x,t)$, i.e., $\widetilde{u}^w(-x,t)=\widetilde{u}^w(x,t)$ for $x\in\mathbb{R},t>0$. Thus we find that 
$\widetilde{u}^w$ is an even function, so that $\widetilde{u}^w_x(0,t)=\widetilde{u}^w_{xxx}(0,t)=0$ holds for $t>0$ since $\widetilde{u}^w_x$ 
and $\widetilde{u}^w_{xxx}$ are odd functions. This implies that $\widetilde{u}^w(x,t)|_{x\ge0}=u^w(x,t)$, where $u^w$ is the solution to 
\eqref{IBVP_u}. Under this observation, using $G(x,t)$ given by \eqref{GreenFunc}, $u^w$ is represented as 
\[
u^w(x,t)=\widetilde{u}^w(x,t)|_{x\ge0}
=-\int_0^t(G(\,\cdot\,,t-\tau)\ast\partial_x(\mathcal{P}f^w)(\,\cdot\,,\tau))(x)\,d\tau.
\]
The integrand $(G(\,\cdot\,,t-\tau)\ast\partial_x(\mathcal{P}f^w)(\,\cdot\,,\tau))(x)$ is transformed into 
\begin{align*}
(G(\,\cdot\,,t-\tau)\ast\partial_x(\mathcal{P}f^w)(\,\cdot\,,\tau))(x)
=&\,\int_{\mathbb{R}}G(x-y,t-\tau)\partial_y(\mathcal{P}f^w)(y,\tau)\,dy \\
=&\,-\int_{\mathbb{R}}\partial_yG(x-y,t-\tau)(\mathcal{P}f^w)(y,\tau)\,dy \\
=&\,\int_{\mathbb{R}}\partial_xG(x-y,t-\tau)(\mathcal{P}f^w)(y,\tau)\,dy. 
\end{align*}
As a result, we have
\[
u^w(x,t)=-\int_0^t\int_{\mathbb{R}}\partial_xG(x-y,t-\tau)(\mathcal{P}f^w)(y,\tau)\,dy\,d\tau.
\]
Recalling $w(x,t)=u^w(x,t)+U^w(x,t)$, we arrive at the integral equation for $w$ as follows:
\begin{equation}\label{SSP_int-eq}
w(x,t)=-\int_0^t\int_{\mathbb{R}}\partial_xG(x-y,t-\tau)(\mathcal{P}f^w)(y,\tau)\,dy\,d\tau+U^w(x,t)
\end{equation}
This leads us to the main theorem. 

%------------------------------
\begin{theorem}\label{thm:existence-selfsimilar}
Let $\gamma\in(0,1)$ and $M>0$. Then there exist $\beta_\ast>0$ such that 
if $\beta\in(0,\beta_\ast)$, then there exists a unique 
\[
w\in B_{\frac{2+\gamma}4}((0,T];BUC^{3+\gamma}(\mathbb{R}_+))\cap\,C_{\frac{2+\gamma}4}^{\frac{1+\gamma}4}((0,T];BUC^2(\mathbb{R}_+)),
\]
which solves \eqref{SSP_int-eq} and satisfies 
\begin{align*}
&\|w_x-U^w_x\|_{B([0,T],BUC(\mathbb{R}_+))}\le\tan\beta, \\
&\|w\|_{B_{\frac{2+\gamma}4}((0,T];BUC^{3+\gamma}(\mathbb{R}_+))}
+\|w\|_{C_{\frac{2+\gamma}4}^{\frac{1+\gamma}4}((0,T];BUC^2(\mathbb{R}_+))}\le M.
\end{align*}
\end{theorem}
%------------------------------

\noindent
Note that this theorem implies the existence and the uniqueness of a bounded self-similar solution 
in the sense of the mild solution. 

%%%%%%%%%%%%%%%
\begin{comment}

%------------------------------
\begin{theorem}\label{thm:existence}
Let $\gamma\in(0,1)$ and assume that $a\in BUC^1( \mathbb{R}_+)$. 
Then there exist $M_0>0$, $\beta _0>0$, and $\delta _0>0$ such that 
if $\beta\in(0,\beta _0)$ and $\delta\in(0,\delta_0)$, then there exists a unique 
\[
w\in B_{\frac{2+\gamma}4}((0,T];BUC^{3+\gamma}(\mathbb{R}_+))\cap\,C_{\frac{2+\gamma}4}^{\frac{1+\gamma}4}((0,T];BUC^2(\mathbb{R}_+)),
\]
which solves
\begin{align}
w(x,t)=&\,\int_0^\infty K(x,y,t)a(y)\,dy-\int_0^t\int_{\mathbb{R}}\partial_xG(x-y,t-\tau)(\mathcal{P}f^w)(y,\tau)\,dy\,d\tau \notag \\
&\,+U^w(x,t) \label{int-eq}
\end{align}
and satisfies 
\begin{align*}
&\|w_x-U^w_x\|_{B([0,T],BUC(\mathbb{R}_+))}\le\delta, \\
&\|w\|_{B_{\frac{2+\gamma}4}((0,T];BUC^{3+\gamma}(\mathbb{R}_+))}
+\|w\|_{C_{\frac{2+\gamma}4}^{\frac{1+\gamma}4}((0,T];BUC^2(\mathbb{R}_+))}\le M_0.
\end{align*}
\end{theorem}
%------------------------------

\end{comment}
%%%%%%%%%%%%%%%

\bigskip\noindent
For the proof of this theorem, we prepare for several lemmas below. In order to state the first lemma, 
let $\rho\in C_0^{\infty}(\mathbb{R})$ be an even function and satisfy 
\[
\text{supp}\,\rho\subset(-1,1), \quad 0\le\rho\le1, \quad \int_{\mathbb{R}}\rho(x)\,dx=1.
\]
For $\rho_\tau(x)=\tau^{-\frac14}\rho(x/\tau^{\frac14})$, we define $F_\tau(x)$ as 
\[
F_\tau(x):=\int_{\mathbb{R}}F(x-y)\rho_\tau(y)\,dy=\int_{\mathbb{R}}F(y)\rho_\tau(x-y)\,dy.
\]

%------------------------------
\begin{lemma}\label{lem:tau-ineq}
Let $\gamma\in(0,1)$ and $\tau>0$. Assume that $F\in BUC^{\gamma}(\mathbb{R})$. Then the inequalities 
\[
\|F_\tau-F\|_\infty\le c_{0,\gamma}[F]_{\gamma}\,\tau^{\frac{\gamma}4}, \quad
\|\partial_xF_\tau\|_\infty\le c_{1,\gamma}[F]_{\gamma}\,\tau^{\frac{\gamma-1}4}
\]
hold, where 
\[
c_{k,\gamma}=\int_{\mathbb{R}}|z|^{\gamma}|\rho^{(k)}(z)|\,dz \quad (k=0,1).
\]
\end{lemma}
%------------------------------

\begin{proof}
From the definition of $F_\tau(x)$, it follows that for $x\in\mathbb{R}$  
\begin{align*}
|F_\tau(x)-F(x)|
\le&\,\int_{\mathbb{R}}|F(x-y)-F(x)|\rho_\tau(y)\,dy
\le[F]_{\gamma}\int_{\mathbb{R}}|y|^\gamma\rho_\tau(y)\,dy \\
%=&\,[F]_{\gamma}\,\tau^{-\frac14}\int_{\mathbb{R}}|y|^\gamma\rho\Bigl(\frac{y}{\tau^{\frac14}}\Bigr)\,dy \\
=&\,[F]_{\gamma}\,\int_{\mathbb{R}}|\tau^{\frac14}z|^\gamma\rho(z)\,dz
=c_{0,\gamma}[F]_{\gamma}\,\tau^{\frac{\gamma}4}.
\end{align*}
This implies the first inequality. Furthermore, using the second equality in the definition of $F_\tau(x)$, 
we see that 
\begin{align*}
\partial_xF_\tau(x)
=&\,\int_{\mathbb{R}}F(y)\partial_x\rho_\tau(x-y)\,dy 
=\tau^{-\frac12}\int_{\mathbb{R}}F(y)\rho'((x-y)/\tau^{\frac14})\,dy \\
=&\,\tau^{-\frac12}\int_{\mathbb{R}}F(x-z)\rho'(z/\tau^{\frac14})\,dz.
\end{align*}
Taking account of 
\[
\int_{\mathbb{R}}\rho'(z/\tau^{\frac14})\,dz=0,
\]
we have
\[
\partial_xF_\tau(x)
=\tau^{-\frac12}\int_{\mathbb{R}}\{F(x-z)-F(x)\}\rho'(z/\tau^{\frac14})\,dz.
\]
This yield that 
\begin{align*}
|\partial_xF_\tau(x)|
\le&\,\tau^{-\frac12}\int_{\mathbb{R}}|F(x-z)-F(x)||\rho'(z/\tau^{\frac14})|\,dz \\
\le&\,[F]_{\gamma}\,\tau^{-\frac12}\int_{\mathbb{R}}|z|^{\gamma}|\rho'(z/\tau^{\frac14})|\,dz \\
=&\,[F]_{\gamma}\,\tau^{-\frac14}\int_{\mathbb{R}}|\tau^{\frac14}\xi|^{\gamma}|\rho'(\xi)|\,d\xi
=c_{1,\gamma}[F]_{\gamma}\,\tau^{\frac{\gamma-1}4},
\end{align*}
so that we obtain the second inequality. 
\end{proof}

%------------------------------
\begin{remark}\label{rem:tau-ineq}
If $F\in W^{1,\infty}(\mathbb{R})$, Lemma \ref{lem:tau-ineq} also holds for $\gamma=1$. 
\end{remark}
%------------------------------

Here, we introduce a class of intermediate spaces between $BUC(\mathbb{R})$ and $BUC^4(\mathbb{R})$. 
For $\alpha\in(0,1)$, define the space $D_G(\alpha,\infty)$ as 
\begin{align*}
&D_G(\alpha,\infty)
:=\{f\in BUC(\mathbb{R})\,|\, %\\
%&\hspace*{75pt} 
\psi(r):=r^{1-\alpha}\|\partial_x^4G(\,\cdot\,,r)\ast f\|_{BUC(\mathbb{R})}\in B(0,1)\}, \\
&\|f\|_{D_G(\alpha,\infty)}=\|f\|_\infty+\|\psi\|_{B(0,1)}.
\end{align*}
Referring to \cite[Proposition 2.2.2, Theorem 1.2.17]{lun;95;book} or \cite[Proposition 6.2, Example 5.15]{lun;18;book}, 
the space $D_G(\alpha,\infty)$ is characterized by
\begin{equation}\label{interpolation}
D_G(\alpha,\infty)=(BUC(\mathbb{R}),BUC^4(\mathbb{R}))_{\alpha,\infty}=BUC^{4\alpha}(\mathbb{R})
\end{equation}
provided that $4\alpha$ is not an integer, where $(X,Y)_{\alpha,\infty}$ is a real interpolation space between 
the Banach spaces $X$ and $Y$. 

Set 
\begin{align*}
v(x,t,\tau)
:=&\,\bigl(\partial_xG(\,\cdot\,,t-\tau)\ast(\mathcal{P}f^w)(\,\cdot\,,\tau)\bigr)(x) \\
\,=&\,\int_{\mathbb{R}}\partial_xG(x-y,t-\tau)(\mathcal{P}f^w)(y,\tau)\,dy.
\end{align*}
For $v(x,t,\tau)$, we establish the following lemmas. 

%------------------------------
\begin{lemma}\label{lem:v-ineq-1}
There exist $C_1,C_2>0$ such that for $k\in\{0,1,2,3\}$
\begin{align}
&\frac1{T^{\frac{3-k+\gamma}4}}\sup_{t\in(0,T]}t^{\frac{2+\gamma}4}\int_0^t\|\partial_x^kv(\,\cdot\,,t,\tau)\|_\infty\,d\tau
\le C_1\|f^w\|_{B_{\frac{2+\gamma}4}((0,T];BUC^{\gamma}(\mathbb{R}_+))}, \label{kth-v-ineq} \\
&\sup_{t\in(0,T]}t^{\frac{2+\gamma}4}\int_0^t[\partial_x^3v(\,\cdot\,,t,\tau)]_\gamma\,d\tau
\le C_2\|f^w\|_{B_{\frac{2+\gamma}4}((0,T];BUC^{\gamma}(\mathbb{R}_+))}. \label{3rd-v-Holder-ineq}
\end{align}
\end{lemma}
%------------------------------

\begin{proof}
Let us prove \eqref{kth-v-ineq}. 
For $k\in\mathbb{N}\cup\{0\}$, $\partial_x^kv(x,t,\tau)$ is transformed into 
\begin{align}
\partial_x^kv(x,t,\tau)
=&\,\bigl(\partial_x^{k+1}G(\,\cdot\,,t-\tau)\ast(\mathcal{P}f^w)(\,\cdot\,,\tau)\bigr)(x) \notag \\
%=&\,\int_{\mathbb{R}}\partial_x^{k+1}G(x-y,t-\tau)(\mathcal{P}f^w)(y,\tau)\,dy \notag \\
=&\,\int_{\mathbb{R}}\partial_x^{k+1}G(x-y,t-\tau)\bigl\{(\mathcal{P}f^w)(y,\tau)-(\mathcal{P}f^w)_{t-\tau}(y,\tau)\bigr\}\,dy \notag \\
&\,+\int_{\mathbb{R}}\partial_x^{k+1}G(x-y,t-\tau)(\mathcal{P}f^w)_{t-\tau}(y,\tau)\,dy \notag \\
=&\,\int_{\mathbb{R}}\partial_x^{k+1}G(x-y,t-\tau)\bigl\{(\mathcal{P}f^w)(y,\tau)-(\mathcal{P}f^w)_{t-\tau}(y,\tau)\bigr\}\,dy \notag \\
&\,+(-1)^{k+1}\int_{\mathbb{R}}\partial_y^{k+1}G(x-y,t-\tau)(\mathcal{P}f^w)_{t-\tau}(y,\tau)\,dy \notag \\
=&\,\int_{\mathbb{R}}\partial_x^{k+1}G(x-y,t-\tau)\bigl\{(\mathcal{P}f^w)(y,\tau)-(\mathcal{P}f^w)_{t-\tau}(y,\tau)\bigr\}\,dy \notag \\
&\,+(-1)^k\int_{\mathbb{R}}\partial_y^kG(x-y,t-\tau)\partial_y(\mathcal{P}f^w)_{t-\tau}(y,\tau)\,dy. \label{kthD-v}
\end{align}
By virtue of \eqref{kthD-v}, Lemma \ref{lem:int_DG} with $\ell=0$, Lemma \ref{lem:tau-ineq}, and \eqref{Pf-Holder}, we have 
\begin{align*}
|\partial_x^kv(x,t,\tau)|
\lesssim&\,(t-\tau)^{-\frac{k+1}4}[\mathcal{P}f^w(\,\cdot\,,\tau)]_\gamma(t-\tau)^{\frac{\gamma}4} \\
&\,+(t-\tau)^{-\frac{k}4}[\mathcal{P}f^w(\,\cdot\,,\tau)]_\gamma(t-\tau)^{\frac{\gamma-1}4} \\
\lesssim&\,(t-\tau)^{\frac{-k-1+\gamma}4}[f^w(\,\cdot\,,\tau)]_\gamma \\
\lesssim&\,(t-\tau)^{\frac{-k-1+\gamma}4}\tau^{-\frac{2+\gamma}4}\|f^w\|_{B_{\frac{2+\gamma}4}((0,T];BUC^{\gamma}(\mathbb{R}_+))}.
\end{align*}
Since 
\begin{align*}
\int_0^t(t-\tau)^{\frac{-k-1+\gamma}4}\tau^{-\frac{2+\gamma}4}\,d\tau
=&\,t^{\frac{-k-1+\gamma}4-\frac{2+\gamma}4+1}B\Bigl(\frac{2-\gamma}4,\frac{3-k+\gamma}4\Bigr) \\
\lesssim&\,t^{-\frac{2+\gamma}4}T^{\frac{3-k+\gamma}4}, 
\end{align*}
for $k\in\{0,1,2,3\}$ and $0<t\le T$, we can confirm \eqref{kth-v-ineq}. 

Let us derive \eqref{3rd-v-Holder-ineq}. Applying \eqref{interpolation} with $\alpha=\gamma/4$, we see that 
\[
[\partial_x^3v(\,\cdot\,,t,\tau)]_\gamma\lesssim\|\partial_x^3v(\,\cdot\,,t,\tau)\|_{D_G(\frac{\gamma}4,\infty)}, 
\]
where we use $\|\partial_x^3v(\,\cdot\,,t,\tau)\|_\infty/T^{\frac{\gamma}4}$ instead of $\|\partial_x^3v(\,\cdot\,,t,\tau)\|_\infty$. 
It follows from \eqref{kthD-v}, Lemma \ref{lem:int_DG} with $(k,\ell)=(8,0),(7,0)$, 
Lemma \ref{lem:tau-ineq}, and \eqref{Pf-Holder} that for $0<r\le1$
\begin{align*}
&\hspace*{-8pt}
\bigl|\partial_x^4G(\,\cdot\,,r)\ast\partial_x^3v(\,\cdot\,,t,\tau)\bigr| \\
=&\,\bigl|\partial_x^4G(\,\cdot\,,r)\ast\bigl(\partial_x^4G(\,\cdot\,,t-\tau)\ast(\mathcal{P}f^w)(\,\cdot\,,\tau)\bigr)\bigr| \\
%=&\,\bigl|\bigl(\partial_x^4G(\,\cdot\,,r)\ast\partial_x^4G(\,\cdot\,,t-\tau)\bigr)\ast(\mathcal{P}f^w)(\,\cdot\,,\tau)\bigr| \\
=&\,\bigl|\partial_x^8G(\,\cdot\,,t+r-\tau)\ast(\mathcal{P}f^w)(\,\cdot\,,\tau)\bigr| \\
\le&\,\int_{\mathbb{R}}\bigl|\partial_x^8G(x-y,t+r-\tau)\bigr|\bigl|(\mathcal{P}f^w)(y,\tau)-(\mathcal{P}f^w)_{t+r-\tau}(y,\tau)\bigr|\,dy \\
&\,+\int_{\mathbb{R}}\bigl|\partial_y^7G(x-y,t+r-\tau)\bigr|\bigl|\partial_y(\mathcal{P}f^w)_{t+r-\tau}(y,\tau)\bigr|\,dy \\
\lesssim&\,(t+r-\tau)^{\frac{\gamma}4-2}\tau^{-\frac{2+\gamma}4}\|f^w\|_{B_{\frac{2+\gamma}4}((0,T];BUC^{\gamma}(\mathbb{R}_+))}.
\end{align*}
Here we obtain 
\begin{align*}
&\hspace*{-8pt}
\int_0^t(t+r-\tau)^{\frac{\gamma}4-2}\tau^{-\frac{2+\gamma}4}\,d\tau \\
=&\,\int_0^{\frac{t}2}(t+r-\tau)^{\frac{\gamma}4-2}\tau^{-\frac{2+\gamma}4}\,d\tau
+\int_{\frac{t}2}^t(t+r-\tau)^{\frac{\gamma}4-2}\tau^{-\frac{2+\gamma}4}\,d\tau \\
\le&\,\Bigl(\frac{t}2+r\Bigr)^{\frac{\gamma}4-2}\int_0^{\frac{t}2}\tau^{-\frac{2+\gamma}4}\,d\tau
+\Bigl(\frac{t}2\Bigr)^{-\frac{2+\gamma}4}\int_{\frac{t}2}^t(t+r-\tau)^{\frac{\gamma}4-2}\,d\tau \\
\lesssim&\,r^{\frac{\gamma}4-1}t^{-1}\Bigl(\frac{t}2\Bigr)^{1-\frac{2+\gamma}4}
+t^{-\frac{2+\gamma}4}\Bigl\{r^{\frac{\gamma}4-1}-\Bigl(\frac{t}2+r\Bigr)^{\frac{\gamma}4-1}\Bigr\} \\
\lesssim&\,r^{\frac{\gamma}4-1}t^{-\frac{2+\gamma}4}.
\end{align*}
This implies 
\begin{align*}
&\sup_{t\in(0,T]}t^{\frac{2+\gamma}4}\int_0^t\bigl(
\sup_{0<r\le1}r^{1-\frac{\gamma}4}\bigl\|\partial_x^4G(\,\cdot\,,r)\ast\partial_x^3v(\,\cdot\,,t,\tau)\bigr\|_{BUC(\mathbb{R})}\bigr)\,d\tau \\
&\lesssim\|f^w\|_{B_{\frac{2+\gamma}4}((0,T];BUC^{\gamma}(\mathbb{R}_+))}.
\end{align*}
Combining this and \eqref{kth-v-ineq} with $k=3$, we find that 
\[
\sup_{t\in(0,T]}t^{\frac{2+\gamma}4}\int_0^t\|\partial_x^3v(\,\cdot\,,t,\tau)\|_{D_G(\frac{\gamma}4,\infty)}\,d\tau
\lesssim\|f^w\|_{B_{\frac{2+\gamma}4}((0,T];BUC^{\gamma}(\mathbb{R}_+))}.
\]
Consequently, we can verify \eqref{3rd-v-Holder-ineq}. 
\end{proof}

%------------------------------
\begin{lemma}\label{lem:v-ineq-2}
There exists $C>0$ such that 
\begin{align*}%\label{v-Holder-ineq}
&\sup_{\begin{subarray}{c}s,t\in(0,T] \\ t>s\end{subarray}}s^{\frac{2+\gamma}4}\frac{\,\displaystyle\biggl\|\int_0^tv(\,\cdot\,,t,\tau)\,d\tau
-\int_0^sv(\,\cdot\,,s,\tau)\,d\tau\biggr\|_{BUC^2(\mathbb{R}_+)}\,}{(t-s)^{\frac{1+\gamma}4}} \\
&\ \le C\|f^w\|_{B_{\frac{2+\gamma}4}((0,T];BUC^{\gamma}(\mathbb{R}_+))}. \notag
\end{align*}
\end{lemma}
%------------------------------

\begin{proof}
We first observe that 
\begin{align*}
&\hspace*{-8pt}
\int_0^t\partial_x^kv(x,t,\tau)\,d\tau-\int_0^s\partial_x^kv(x,s,\tau)\,d\tau \\
=&\,\int_s^t\partial_x^kv(x,t,\tau)\,d\tau+\int_0^s\bigl\{\partial_x^kv(x,t,\tau)-\partial_x^kv(x,s,\tau)\bigr\}\,d\tau
\end{align*}
As for the first term, it follows from an argument similar to the proof of \eqref{kth-v-ineq} that 
\[
|\partial_x^kv(x,t,\tau)|
\lesssim(t-\tau)^{\frac{-k-1+\gamma}4}\tau^{-\frac{2+\gamma}4}\|f^w\|_{B_{\frac{2+\gamma}4}((0,T];BUC^{\gamma}(\mathbb{R}_+))}.
\]
This implies that for $k\in\{0,1,2\}$ 
\begin{align*}
\int_s^t|\partial_x^kv(x,t,\tau)|\,d\tau
\lesssim&\,\int_s^t(t-\tau)^{\frac{-k-1+\gamma}4}\tau^{-\frac{2+\gamma}4}\,d\tau\,\|f^w\|_{B_{\frac{2+\gamma}4}((0,T];BUC^{\gamma}(\mathbb{R}_+))} \\
%\lesssim&\,s^{-\frac{2+\gamma}4}\int_s^t(t-\tau)^{\frac{-k-1+\gamma}4}\,d\tau\,\|f^w\|_{B_{\frac{2+\gamma}4}((0,T];BUC^{\gamma}(\mathbb{R}_+))} \\
\lesssim&\,s^{-\frac{2+\gamma}4}(t-s)^{\frac{3-k+\gamma}4}\,\|f^w\|_{B_{\frac{2+\gamma}4}((0,T];BUC^{\gamma}(\mathbb{R}_+))} \\
\lesssim&\,s^{-\frac{2+\gamma}4}(t-s)^{\frac{1+\gamma}4}\,\|f^w\|_{B_{\frac{2+\gamma}4}((0,T];BUC^{\gamma}(\mathbb{R}_+))}\,T^{\frac{2-k}4}.
\end{align*}
Let us consider the second term. From the definition of $v(x,t,\tau)$ and 
\[
\partial_x^{k+1}G(\,\cdot\,,t-\tau)-\partial_x^{k+1}G(\,\cdot\,,s-\tau)=\int_{s-\tau}^{t-\tau}\partial_\sigma\partial_x^{k+1}G(\,\cdot\,,\sigma)\,d\sigma, 
\]
$\partial_x^kv(x,t,\tau)-\partial_x^kv(x,s,\tau)$ is transformed into
\begin{align*}
&\hspace*{-8pt}
\partial_x^kv(x,t,\tau)-\partial_x^kv(x,s,\tau) \\
=&\,\int_{\mathbb{R}}\biggl(\int_{s-\tau}^{t-\tau}\partial_\sigma\partial_x^{k+1}G(x-y,\sigma)\,d\sigma\biggr)(\mathcal{P}f^w)(y,\tau)\,dy \\
=&\,\int_{s-\tau}^{t-\tau}\int_{\mathbb{R}}\partial_\sigma\partial_x^{k+1}G(x-y,\sigma)
\bigl\{(\mathcal{P}f^w)(y,\tau)-(\mathcal{P}f^w)_\sigma(y,\tau)\bigr\}\,dyd\sigma \\
&\,+\int_{s-\tau}^{t-\tau}\int_{\mathbb{R}}\partial_\sigma\partial_x^{k+1}G(x-y,\sigma)(\mathcal{P}f^w)_\sigma(y,\tau)\,dyd\sigma \\
=&\,\int_{s-\tau}^{t-\tau}\int_{\mathbb{R}}\partial_\sigma\partial_x^{k+1}G(x-y,\sigma)
\bigl\{(\mathcal{P}f^w)(y,\tau)-(\mathcal{P}f^w)_\sigma(y,\tau)\bigr\}\,dyd\sigma \\
&\,+(-1)^k\int_{s-\tau}^{t-\tau}\int_{\mathbb{R}}\partial_\sigma\partial_y^kG(x-y,\sigma)\partial_y(\mathcal{P}f^w)_\sigma(y,\tau)\,dyd\sigma.
\end{align*}
Applying Lemma \ref{lem:int_DG} with $\ell=1$, Lemma \ref{lem:tau-ineq}, and \eqref{Pf-Holder}, we have 
\begin{align*}
&\hspace*{-8pt}
|\partial_x^kv(x,t,\tau)-\partial_x^kv(x,s,\tau)| \\
\lesssim&\,\biggl\{\int_{s-\tau}^{t-\tau}\bigl(\sigma^{-\frac{4+(k+1)}4+\frac{\gamma}4}+\sigma^{-\frac{4+k}4+\frac{\gamma-1}4}\bigr)\,d\sigma\biggr\}\,
\tau^{-\frac{2+\gamma}4}\|f^w\|_{B_{\frac{2+\gamma}4}((0,T];BUC^{\gamma}(\mathbb{R}_+))}.
\end{align*}
For $k\in\{0,1\}$, it follows that 
\begin{align*}
&\hspace*{-8pt}
\int_0^s\biggl(\int_{s-\tau}^{t-\tau}\sigma^{\frac{\gamma-(k+1)}4-1}\,d\sigma\biggr)\,\tau^{-\frac{2+\gamma}4}\,d\tau \\
=&\,\int_0^s\biggl(\int_{s-\tau}^{t-\tau}\sigma^{\frac{1+\gamma}4-1}\,d\sigma\biggr)\,(s-\tau)^{-\frac{k+2}4}\tau^{-\frac{2+\gamma}4}\,d\tau \\
%\lesssim&\,\int_0^s\bigl\{(t-\tau)^{\frac{1+\gamma}4}-(s-\tau)^{\frac{1+\gamma}4}\bigr\}\,(s-\tau)^{-\frac{k+2}4}\tau^{-\frac{2+\gamma}4}\,d\tau \\
\lesssim&\,(t-s)^{\frac{1+\gamma}4}\int_0^s(s-\tau)^{-\frac{k+2}4}\tau^{-\frac{2+\gamma}4}\,d\tau \\
=&\,(t-s)^{\frac{1+\gamma}4}s^{-\frac{k+2}4-\frac{2+\gamma}4+1}B\Bigl(\frac{2-\gamma}4,\frac{2-k}4\Bigr) \\
\lesssim&\,(t-s)^{\frac{1+\gamma}4}s^{-\frac{2+\gamma}4}T^{\frac{2-k}4}.
\end{align*}
For $k=2$, we see that 
\begin{align*}
&\hspace*{-8pt}
\int_0^s\biggl(\int_{s-\tau}^{t-\tau}\sigma^{\frac{\gamma-3}4-1}\,d\sigma\biggr)\,\tau^{-\frac{2+\gamma}4}\,d\tau \\
=&\,\frac4{3-\gamma}\int_0^s\bigl\{(s-\tau)^{\frac{\gamma-3}4}-(t-\tau)^{\frac{\gamma-3}4}\bigr\}\,\tau^{-\frac{2+\gamma}4}\,d\tau \\
=&\,\frac4{3-\gamma}\biggl\{\int_0^{\frac{s}2}\bigl\{(s-\tau)^{\frac{\gamma-3}4}-(t-\tau)^{\frac{\gamma-3}4}\bigr\}\,\tau^{-\frac{2+\gamma}4}\,d\tau \\
&\,\hspace*{30pt}+\int_{\frac{s}2}^s\bigl\{(s-\tau)^{\frac{\gamma-3}4}-(t-\tau)^{\frac{\gamma-3}4}\bigr\}\,\tau^{-\frac{2+\gamma}4}\,d\tau\biggr\} \\
=:&\,\frac4{3-\gamma}(I_1+I_2). 
\end{align*}
The estimates of $I_1$ and $I_2$ are given by 
\begin{align*}
I_1=&\,\int_0^{\frac{s}2}(s-\tau)^{-\frac{3-\gamma}4}(t-\tau)^{-\frac{3-\gamma}4}
\bigl\{(t-\tau)^{\frac{3-\gamma}4}-(s-\tau)^{\frac{3-\gamma}4}\bigr\}\,\tau^{-\frac{2+\gamma}4}\,d\tau \\
\le&\,\Bigl(\frac{s}2\Bigr)^{-\frac{3-\gamma}4}\Bigl(t-\frac{s}2\Bigr)^{-\frac{3-\gamma}4}(t-s)^{\frac{3-\gamma}4}
\int_0^{\frac{s}2}\tau^{-\frac{2+\gamma}4}\,d\tau \\
\lesssim&\,\Bigl(\frac{s}2\Bigr)^{-1}\Bigl(\frac{t}2\Bigr)^{-\frac{2-2\gamma}4}(t-s)^{\frac{1+\gamma}4}t^{\frac{2-2\gamma}4}
\Bigl(\frac{s}2\Bigr)^{1-\frac{2+\gamma}4} \\
\lesssim&\,s^{-\frac{2+\gamma}4}(t-s)^{\frac{1+\gamma}4}, \\
I_2\le&\,\Bigl(\frac{s}2\Bigr)^{-\frac{2+\gamma}4}\int_{\frac{s}2}^s\bigl\{(s-\tau)^{\frac{\gamma-3}4}-(t-\tau)^{\frac{\gamma-3}4}\bigr\}\,d\tau \\
\lesssim&\,s^{-\frac{2+\gamma}4}\Bigr[(t-s)^{\frac{1+\gamma}4}
-\Bigl\{\Bigl(t-\frac{s}2\Bigr)^{\frac{1+\gamma}4}-\Bigl(\frac{s}2\Bigr)^{\frac{1+\gamma}4}\Bigr\}\Bigr] \\
\lesssim&\,s^{-\frac{2+\gamma}4}(t-s)^{\frac{1+\gamma}4}.
\end{align*}
In the estimate of $I_1$, we have used
\[
\Bigl(\frac{s}2\Bigr)^{-\frac{3-\gamma}4}\Bigl(t-\frac{s}2\Bigr)^{-\frac{3-\gamma}4}
\le\Bigl(\frac{s}2\Bigr)^{-1}\Bigl(\frac{s}2\Bigr)^{\frac{1+\gamma}4}\Bigl(\frac{t}2\Bigr)^{-\frac{3-\gamma}4}
\le\Bigl(\frac{s}2\Bigr)^{-1}\Bigl(\frac{t}2\Bigr)^{-\frac{2-2\gamma}4}. 
\]
These estimates yield the desired result. 
\end{proof}

\smallskip
We incorporate the estimate of $\|f^w\|_{B_{\frac{2+\gamma}4}((0,T];BUC^{\gamma}(\mathbb{R}_+))}$ in the appendix. 
We are ready to prove Theorem \ref{thm:existence-selfsimilar}. 

%\begin{proof}
\bigskip\noindent
\textit{Proof of Theorem \ref{thm:existence-selfsimilar}.}
For $\beta\in(0,\pi/2)$ and $M>0$, set 
\begin{align*}
\mathcal{Z}_{\beta,M}
:=&\,\{w\in B_{\frac{2+\gamma}4}((0,T];BUC^{3+\gamma}(\mathbb{R}_+))\cap C_{\frac{2+\gamma}4}^{\frac{1+\gamma}4}((0,T];BUC^2(\mathbb{R}_+))\,| \\
&\hspace*{8.5pt} 
w(x,0)\equiv0,\ \|w_x-U^w_x\|_{B([0,T],BUC( \mathbb{R}_+))}\le\tan\beta, \\
&\hspace*{8.5pt} 
\|w\|_{B_{\frac{2+\gamma}4}((0,T];BUC^{3+\gamma}(\mathbb{R}_+))}+\|w\|_{C_{\frac{2+\gamma}4}^{\frac{1+\gamma}4}((0,T];BUC^2(\mathbb{R}_+))}\le M\}.
\end{align*}
Then we define the mapping $\mathcal{F}$ on $\mathcal{Z}_{\beta,M}$ as 
\begin{equation*}
(\mathcal{F}w)(x,t):=-\int_0^t\int_{\mathbb{R}}\partial_xG(x-y,t-\tau)(\mathcal{P}f^w)(y,\tau)\,dyd\tau+U^w(x,t)
\end{equation*}
for $w\in\mathcal{Z}_{\beta,M}$. We show that $\mathcal{F}$ is a contraction mapping on $\mathcal{Z}_{\beta,M}$ 
for a suitable constant $\beta>0$, which implies that $\mathcal{F}$ has a unique fixed-point $w\in\mathcal{Z}_{\beta,M}$ 
by Banach's fixed-point theorem. 

Let us prove that $\mathcal{F}$ maps $\mathcal{Z}_{\beta,M}$ into itself. Set 
\[
\|w\|_{\mathcal{Z}_{\beta,M}}
:=\|w\|_{B_{\frac{2+\gamma}4}((0,T];BUC^{3+\gamma}(\mathbb{R}_+))}
+\|w\|_{C_{\frac{2+\gamma}4}^{\frac{1+\gamma}4}((0,T];BUC^2(\mathbb{R}_+))}.
\]
Taking account of 
\begin{align}
\|w_x\|_{B([0,T];BUC(\mathbb{R}_+))}
\le&\,\|w_x-U^w_x\|_{B([0,T];BUC(\mathbb{R}_+))}+\|U^w_x\|_{B([0,T];BUC(\mathbb{R}_+))} \notag \\
\le&\,C_M\tan\beta \label{est-w_x}
\end{align}
for a constant $C_M>0$, it follows from Proposition \ref{prop:linear-est}, Lemma \ref{lem:v-ineq-1}, Lemma \ref{lem:v-ineq-2}, 
and Lemma \ref{lem:f-Holder} that 
\begin{align*}
&\|\mathcal{F}w\|_{\mathcal{Z}_{\beta,M}}
\le C_M(\|w_x\|_{B([0,T];BUC(\mathbb{R}_+))}^2+\tan\beta) %\\
%&\hspace*{45pt}
\le C_{M,1}(1+\tan\beta)\tan\beta, \\
&\|(\mathcal{F}w)_x-U^w_x\|_{B([0,T];BUC(\mathbb{R}_+))}
\le C_M\|w_x\|_{B([0,T];BUC(\mathbb{R}_+))}^2 %\\
%&\hspace*{134pt}
\le C_{M,2}\tan^2\beta
\end{align*}
for some constants $C_M>0$ and $C_{M,i}>0\,(i=1,2)$. These inequalities yield that 
\[
\|\mathcal{F}w\|_{\mathcal{Z}_{\beta,M}}\le M, \quad \|(\mathcal{F}w)_x-U^w_x\|_{B([0,T];BUC(\mathbb{R}_+))}\le\tan\beta
\]
provided that 
\[
0<\tan\beta<\min\biggl\{1,\frac{M}{2C_{M,1}},\frac1{C_{M,2}}\biggr\}. 
\]

Let us prove that $\mathcal{F}$ is a contraction mapping on $\mathcal{Z}_{\beta,M}$. 
Take $w^{(j)}\in\mathcal{Z}_{\beta,M}\,(j=1,2)$. Since 
\begin{align*}
&\hspace*{-8pt}
(\mathcal{F}w^{(1)})(x,t)-(\mathcal{F}w^{(2)})(x,t) \\
=&\,-\int_0^t\int_{\mathbb{R}}\partial_xG(x-y,t-\tau)\bigl\{(\mathcal{P}f^{w^{(1)}})(y,\tau)-(\mathcal{P}f^{w^{(2)}})(y,\tau)\bigr\}\,dy\,d\tau \\[0.1cm]
&\,+U^{w^{(1)}}(x,t)-U^{w^{(2)}}(x,t),
\end{align*}
Corollary \ref{cor:differ-linear-est}, Lemma \ref{lem:v-ineq-1}, Lemma \ref{lem:v-ineq-2}, and Lemma \ref{lem:differ-f-Holder} imply that 
there exists $C_M>0$ such that 
\begin{align*}
&\hspace*{-8pt}
\|\mathcal{F}w^{(1)}-\mathcal{F}w^{(2)}\|_{\mathcal{Z}_{\beta,M}} \\
\le&\,C_M\bigl\{\bigl(\max_{j=1,2}\|w^{(j)}_x\|_{B([0,T];BUC(\mathbb{R}_+)}\bigr)
\|w^{(1)}-w^{(2)}\|_{B_{\frac{2+\gamma}4}((0,T];BUC^{3+\gamma}(\mathbb{R}_+))} \\
&\hspace*{20pt} 
+\|b^{w^{(1)}}-b^{w^{(2)}}\|_{C_{\frac{3+\gamma}4}^{\frac{1+\gamma}4}((0,T];\mathbb{R})}\tan\beta\bigr\},
\end{align*}
where $b^{w^{(j)}}(t):=(w^{(j)}_{xx}(0,t))^2$. From the definition of $b^{w^{(j)}}(t)$, we see that 
\begin{align*}
&\hspace*{-8pt}
|b^{w^{(1)}}(t)-b^{w^{(2)}}(t)-(b^{w^{(1)}}(s)-b^{w^{(2)}}(s))| \\
\le&\,\bigl|(w^{(1)}_{xx}(0,t))^2-(w^{(2)}_{xx}(0,t))^2-\{(w^{(1)}_{xx}(0,s))^2-(w^{(2)}_{xx}(0,s))^2\}\bigr| \\
%\le&\,\bigl|w^{(1)}_{xx}(0,t)+w^{(2)}_{xx}(0,t)\bigr|\bigl|w^{(1)}_{xx}(0,t)-w^{(2)}_{xx}(0,t)-\bigl(w^{(1)}_{xx}(0,s)-w^{(2)}_{xx}(0,s)\bigr)\bigr| \\
%&\,+\bigl|\bigl(w^{(1)}_{xx}(0,t)+w^{(2)}_{xx}(0,t)\bigr)-\bigl(w^{(1)}_{xx}(0,s)+w^{(2)}_{xx}(0,s)\bigr)\bigr|\bigl|w^{(1)}_{xx}(0,s)-w^{(2)}_{xx}(0,s)\bigr| \\
\le&\,\bigl(\bigl|w^{(1)}_{xx}(0,t)\bigr|+\bigl|w^{(2)}_{xx}(0,t)\bigr|\bigr) \\
&\,\bigl|w^{(1)}_{xx}(0,t)-w^{(2)}_{xx}(0,t)-\bigl(w^{(1)}_{xx}(0,s)-w^{(2)}_{xx}(0,s)\bigr)\bigr| \\
&\,+\bigl(\bigl|w^{(1)}_{xx}(0,t)-w^{(1)}_{xx}(0,s)\bigr|+\bigl|w^{(2)}_{xx}(0,t)-w^{(2)}_{xx}(0,s)\bigr|\bigr) %\\
%&\hspace*{15pt}
\bigl|w^{(1)}_{xx}(0,s)-w^{(2)}_{xx}(0,s)\bigr|,
\end{align*}
so that there exists $C_M>0$ such that 
\begin{align*}
&\hspace*{-8pt}
\|b^{w^{(1)}}-b^{w^{(2)}}\|_{C_{\frac{3+\gamma}4}^{\frac{1+\gamma}4}((0,T];\mathbb{R})}
\le C_M\|w^{(1)}-w^{(2)}\|_{C_{\frac{2+\gamma}4}^{\frac{1+\gamma}4}((0,T];BUC^2(\mathbb{R}_+))}.
\end{align*}
Thus we have 
\[
\|\mathcal{F}w^{(1)}-\mathcal{F}w^{(2)}\|_{\mathcal{Z}_{\beta,M}}
\le C_{M,3}\|w^{(1)}-w^{(2)}\|_{\mathcal{Z}_{\beta,M}}\tan\beta
\]
for a constant $C_{M,3}>0$. Consequently, setting 
\[
\tan\beta_\ast:=\min\biggl\{1,\frac{M}{2C_{M,1}},\frac1{C_{M,2}},\frac1{2C_{M,3}}\biggr\}, 
\]
$\mathcal{F}$ is a contraction mapping on $\mathcal{Z}_{\beta,M}$ provided that $\beta\in(0,\beta_\ast)$. 
%\end{proof}

%------------------------------
\begin{remark}\label{rem:existence-selfsimilar}
From the integral equation in Theorem \ref{thm:existence-selfsimilar}, we see that 
\begin{align*}
w(0,t)
=&\,-\int_0^t\int_{\mathbb{R}}\partial_xG(-y,t-\tau)(\mathcal{P}f^w)(y,\tau)\,dy\,d\tau+U^w(0,t) \\
=&\,\int_0^t\int_{\mathbb{R}}\partial_xG(y,t-\tau)(\mathcal{P}f^w)(y,\tau)\,dy\,d\tau+U^w(0,t).
\end{align*}
Since $U^w(0,t)$ is given by 
\begin{align*}
U^w(0,t)=&\,-\frac{2\tan\beta}{\pi}\int_0^\infty\frac{1-e^{-\xi^4t}}{\xi^2}\,d\xi
+\frac{2c_\beta}{\pi}\int_0^t\int_0^\infty b^w(\tau)e^{-\xi^4(t-\tau)}\,d\xi d\tau \\
=&\,-\frac{2\tan\beta}{\pi}\Gamma\Bigl(\frac34\Bigr)t^{\frac14}+\frac{c_\beta}{2\pi}\Gamma\Bigl(\frac14\Bigr)\int_0^tb^w(\tau)(t-\tau)^{-\frac14}\,d\tau,
\end{align*}
and $b^w(t)$ and the first term of the right-hand side of $w(0,t)$ are estimated by 
\begin{align*}
&\|b^w\|_{B_{\frac12}((0,T];\mathbb{R})}
%\le C\|w_x\|_{B((0,T];BUC(\mathbb{R}_+))}\biggl(\frac1{T^{\frac{\gamma}4}}\sup_{t\in(0,T]}t^{\frac{2+\gamma}4}\|w_{xxx}(t)\|_\infty\biggr)
%\le C_M\|w_x\|_{B((0,T];BUC(\mathbb{R}_+))}
\le C_M\tan\beta, \\
&\biggl|\int_0^t\int_{\mathbb{R}}\partial_xG(y,t-\tau)(\mathcal{P}f^w)(y,\tau)\,dy\,d\tau\biggr|\le C_Mt^{\frac14}\tan^2\beta
\end{align*}
by virtue of \eqref{interpolation-1}, \eqref{est-w_x}, Lemma \ref{lem:v-ineq-1}, and Lemma \ref{lem:f-Holder}, we have 
\[
t^{-\frac14}w(0,t)\le -\frac{2\tan\beta}{\pi}\Gamma\Bigl(\frac34\Bigr)+C_M\tan^2\beta
\]
for $t>0$. Thus, if $\beta>0$ is sufficiently small such that 
\[
%C_M\tan^2\beta\le\frac{\tan\beta}{\pi}\Gamma\Bigl(\frac34\Bigr), 
\tan\beta\le\frac1{\pi\,C_M}\Gamma\Bigl(\frac34\Bigr), 
\]
we can conclude that 
\[
t^{-\frac14}w(0,t)\le -\frac{\tan\beta}{\pi}\Gamma\Bigl(\frac34\Bigr)<0.
\]
This argument ensures that a self-similar solution obtained in Theorem \ref{thm:existence-selfsimilar} is not  $w(x,t)\equiv0$ 
as long as $\beta>0$. 
\end{remark}

\medskip
The mild solution $w$ in Theorem \ref{thm:existence-selfsimilar} is not differentiable with respect to $t$. 
Here we define a weak solution to the problem \eqref{eq:161031a2}--\eqref{eq:161031a5}. Set
\begin{align*}
\mathcal{C}_0^{\infty}
:=&\,\{\varphi\in C^{\infty}(\mathbb{R}_+\times[0,T])\mid \\
&\,\sup_{x\in\mathbb{R}_+}(1+x^2)^k|\partial_x^\ell\partial_t^m\varphi(x,t)|<\infty\,(t\in[0,T],k,\ell\in\mathbb{N}\cup\{0\},m\in\{0,1\}), \\
&\,\lim_{t\to T^-}|\varphi(x,t)|=0\ (x\in\mathbb{R}_+)\}
\end{align*}
and define $\langle\,\cdot\,,\,\cdot\,\rangle$ as 
\[
\langle f(\,\cdot\,,t),\varphi(\,\cdot\,,t)\rangle:=\int_0^\infty f(x,t)\varphi(x,t)\,dx
\]
for $f(\,\cdot\,,t)\in L_{loc}^1(\mathbb{R}_+)$ and $\varphi\in\mathcal{C}_0^{\infty}$. 

%------------------------------
\begin{definition}[Weak solution]\label{def:weak-sol}
Let $a\in BUC^1(\mathbb{R}_+)$. We say that 
\[
w\in B_{\frac{2+\gamma}4}((0,T];BUC^{3+\gamma }(\mathbb{R}_+))\cap C_{\frac{2+\gamma}4}^{\frac{1+\gamma}4}((0,T];BUC^2(\mathbb{R}_+))
\]
is a weak solution to the problem \eqref{eq:161031a2}--\eqref{eq:161031a5} if $w$ satisfies 
\begin{align*}
&\hspace*{-8pt}
\int_0^T\langle w(\,\cdot\,,t),\varphi_t(\,\cdot\,,t),\rangle\,dt \\
=&\,-\langle a(\,\cdot\,),\varphi(\,\cdot\,,0)\rangle \\
&\,-\int_0^T\Bigl\langle\dfrac{1}{\bigl(1+(w_x(\,\cdot\,,t))^2\bigr)^{\frac12}}\partial_x\biggl(\dfrac{w_{xx}(\,\cdot\,,t)}{\bigl(1+(w_x(\,\cdot\,,t))^2\bigr)^{\frac32}}\biggr),
\varphi_x(\,\cdot\,,t),\Bigr\rangle\,dt
\end{align*}
for all $\varphi\in\mathcal{C}_0^{\infty}$ and $w_x(0,t)=\tan\beta$ for $t\in(0,T]$. 
\end{definition}
%------------------------------

\noindent
Note that the case where $a(x)\equiv0$ implies the definition of a weak solution to the problem \eqref{SSP}, 
which is the initial boundary value problem for a bounded self-similar solution. 

%------------------------------
\begin{proposition}\label{prop:weak-sol}
Let $w$ be the solution in Theorem \ref{thm:existence-selfsimilar}. Then $w$ is a weak solution to the problem \eqref{SSP}. 
%\eqref{eq:161031a2}--\eqref{eq:161031a5}. 
\end{proposition}
%------------------------------

\begin{proof}
Considering $\langle\,\cdot\,,\varphi_t\rangle$ in the both sides of \eqref{SSP_int-eq} and integrating it with respect to $t$ on the interval $[0,T]$, 
we have 
\begin{align*}
&\hspace*{-12pt}
\int_0^T\langle w(\,\cdot\,,t),\varphi_t(\,\cdot\,,t)\rangle\,dt \\
%=\,&\,\int_0^T\Bigl\langle\int_0^\infty K(\,\cdot\,,y,t)a(y)\,dy,\varphi_t(\,\cdot\,,t)\Bigr\rangle\,dt \\
=\,&\,-\int_0^T\Bigl\langle\int_0^t\int_{\mathbb{R}}\partial_xG(\,\cdot\,-y,t-\tau)(\mathcal{P}f^w)(y,\tau)\,dyd\tau,\varphi_t(\,\cdot\,,t)\Bigr\rangle\,dt \\
&\,+\int_0^T\langle U^w(\,\cdot\,,t),\varphi_t(\,\cdot\,,t)\rangle\,dt \\
=:&\,I_1+I_2. 
\end{align*}
Note that, since the integrals $I_i\,(i=1,2)$ are absolutely integrable from Proposition \ref{prop:linear-est}, Lemma \ref{lem:int_DG}, 
Lemma \ref{lem:tau-ineq}, and $\varphi\in\mathcal{C}_0^{\infty}$, the order of the iterated integrals can be changed. 
%Using the integration by parts, the equation $\partial_tK=-\partial_x^4K$, and the boundary condition $\partial_x^3K(0,y,t)=0$, 
%we obtain  
%\[
%I_1=-\langle a(\,\cdot\,),\varphi(\,\cdot\,,0)\rangle-\int_0^T\Bigl\langle\int_0^\infty\partial_x^3K(\,\cdot\,,y,t)a(y)\,dy,\varphi_x(\,\cdot\,,t)\Bigr\rangle\,dt.
%\]
As for $I_1$, we first observe 
\begin{align*}
&\hspace*{-8pt}
\int_0^T\biggl(\int_0^t\partial_xG(x-y,t-\tau)(\mathcal{P}f^w)(y,\tau)\,d\tau\biggr)\varphi_t(x,t)\,dt \\
=&\,\int_0^T\biggl(\int_\tau^T\partial_xG(x-y,t-\tau)\varphi_t(x,t)\,dt\biggr)(\mathcal{P}f^w)(y,\tau)\,d\tau \\
=&\,-\int_0^T\partial_xG(x-y,0)\varphi(x,\tau)(\mathcal{P}f^w)(y,\tau)\,d\tau \\
&\,-\int_0^T\biggl(\int_\tau^T\partial_t\partial_xG(x-y,t-\tau)\varphi(x,t)\,dt\biggr)(\mathcal{P}f^w)(y,\tau)\,d\tau \\
=&\,-\int_0^T\partial_xG(x-y,0)(\mathcal{P}f^w)(y,\tau)\varphi(x,\tau)\,d\tau \\
&\,-\int_0^T\biggl(\int_0^t\partial_x\partial_tG(x-y,t-\tau)(\mathcal{P}f^w)(y,\tau)\,d\tau\biggr)\varphi(x,t)\,dt \\
=&\,-\int_0^T\partial_xG(x-y,0)(\mathcal{P}f^w)(y,\tau)\varphi(x,\tau)\,d\tau \\
&\,+\int_0^T\biggl(\int_0^t\partial_x^5G(x-y,t-\tau)(\mathcal{P}f^w)(y,\tau)\,d\tau\biggr)\varphi(x,t)\,dt. 
\end{align*}
In the last equality, we have used the equation $\partial_tG=-\partial_x^4G$. Taking account of the fact that 
$G(y,0)(\mathcal{P}f^w)(y,\tau)$ is an odd function with respect to $y$, it follows from the integration by parts that 
\begin{align*}
&\hspace*{-8pt}
\int_0^T\Bigr\langle\int_{\mathbb{R}}\partial_xG(\,\cdot\,-y,0)(\mathcal{P}f^w)(y,\tau)\,dy,\varphi(\,\cdot\,,\tau)\Bigr\rangle\,d\tau \\
=&\,-\int_0^T\Bigl\langle\int_{\mathbb{R}} G(\,\cdot\,-y,0)(\mathcal{P}f^w)(y,\tau)\,dy,\varphi_x(\,\cdot\,,\tau)\Bigr\rangle\,d\tau \\
=&\,-\int_0^T\langle(\mathcal{P}f^w)(\,\cdot\,,\tau),\varphi_x(\,\cdot\,,\tau)\rangle\,d\tau \\
=&\,-\int_0^T\langle f^w(\,\cdot\,,\tau)-f^w(0,\tau),\varphi_x(\,\cdot\,,\tau)\rangle\,d\tau \\
=&\,-\int_0^T\langle f^w(\,\cdot\,,\tau),\varphi_x(\,\cdot\,,\tau)\rangle\,d\tau-\int_0^Tf^w(0,\tau)\varphi(0,\tau)\,d\tau. 
\end{align*}
Furthermore, using the fact that $\partial_x^4G(y,t-\tau)(\mathcal{P}f^w)(y,\tau)$ is an odd function with respect to $y$, 
the integration by parts implies that 
\begin{align*}
&\hspace*{-8pt}
\int_0^T\Bigl\langle\int_0^t\int_{\mathbb{R}}\partial_x^5G(\,\cdot\,-y,t-\tau)(\mathcal{P}f^w)(y,\tau)\,dyd\tau,\varphi(\,\cdot\,,t)\Bigr\rangle\,dt \\
=&\,-\int_0^T\Bigl\langle\int_0^t\int_{\mathbb{R}}\partial_x^4G(\,\cdot\,-y,t-\tau)(\mathcal{P}f^w)(y,\tau)\,dyd\tau,\varphi_x(\,\cdot\,,t)\Bigr\rangle\,dt.
\end{align*}
As a result, we obtain
\begin{align*}
I_1=&\,-\int_0^T\langle f^w(\,\cdot\,,t),\varphi_x(\,\cdot\,,t)\rangle\,dt-\int_0^Tf^w(0,t)\varphi(0,t)\,dt \\
&\,+\int_0^T\Bigl\langle\int_0^t\int_{\mathbb{R}}\partial_x^4G(\,\cdot\,-y,t-\tau)(\mathcal{P}f^w)(y,\tau)\,dyd\tau,\varphi_x(\,\cdot\,,t)\Bigr\rangle\,dt.
\end{align*}
Concerning $I_2$, since $U$ is the solution to \eqref{LP}, we have
\[
I_2=-\int_0^Tc_\beta(w_{xx}(0,t))^2\varphi(0,t)\,dt-\int_0^\infty\langle\partial_x^3U^w(\,\cdot\,,t),\varphi_x(\,\cdot\,,t)\rangle\,dt. 
\]
Consequently, taking account of 
\begin{equation}\label{3rd_Dw}
w_{xxx}(x,t)
%=&\,\int_0^\infty \partial_x^3K(x,y,t)a(y)\,dy
=-\int_0^t\int_{\mathbb{R}}\partial_x^4G(x-y,t-\tau)(\mathcal{P}f^w)(y,\tau)\,dyd\tau
+\partial_x^3U^w(x,t), 
\end{equation}
we see that 
\begin{align*}
%&\hspace*{-8pt}
\int_0^T\langle w(\,\cdot\,,t),\varphi_t(\,\cdot\,,t)\rangle\,dt %\\
%=&\,-\langle a(\,\cdot\,),\varphi(\,\cdot\,,0)\rangle
=&\,-\int_0^T\langle w_{xxx}(\,\cdot\,,t)+f^w(\,\cdot\,,t),\varphi_x(\,\cdot\,,t)\rangle\,dt \\
&\,-\int_0^T\bigl\{f^w(0,t)+c_\beta(w_{xx}(0,t))^2\bigr\}\varphi(0,t)\,dt.
\end{align*}
Recalling \eqref{transformed_RHS}, the definition of $f^w$, and the no-flux condition \eqref{eq:161031a4}, it follows that 
\[
w_{xxx}+f^w=\dfrac{1}{(1+w_x^2)^{\frac12}}\partial_x\biggl(\dfrac{w_{xx}}{(1+w_x^2)^{\frac32}}\biggr), \quad w_{xxx}(0,t)+f^w(0,t)=0.
\]
Furthermore, by virtue of \eqref{3rd_Dw} %$\partial_x^3K(0,y,t)=0$, 
and the fact that $G(y,t-\tau)(\mathcal{P}f^w)(y,\tau)$ is an odd function with respect to $y$, we have
\begin{align*}
w_{xxx}(0,t)=\partial_x^3U^w(0,t)=c_\beta(w_{xx}(0,t))^2.
\end{align*}
These calculations yield the desired result. 
\end{proof}

%------------------------------
\begin{remark}
When we consider the stability of a self-similar solution, we need the existence of a solution to the problem \eqref{eq:161031a2}--\eqref{eq:161031a5}. 
Applying an argument similar to the proof of Theorem \ref{thm:existence-selfsimilar} and Proposition \ref{prop:weak-sol} to the integral equation 
\begin{align*}
w(x,t)=&\,\int_0^\infty K(x,y,t)a(y)\,dy-\int_0^t\int_{\mathbb{R}}\partial_xG(x-y,t-\tau)(\mathcal{P}f^w)(y,\tau)\,dy\,d\tau \notag \\
&\,+U^w(x,t)
\end{align*}
for $a\in BUC^1(\mathbb{R}_+)$, we can prove that a mild solution exists and it is a weak solution as in Definition \ref{def:weak-sol}. 
In the proof of the existence of a mild solution,  we replace $M$ and $\beta_\ast$ in Theorem \ref{thm:existence-selfsimilar} with 
$M_0:=C_0\|a'\|_\infty$ and $\beta_0$, where $C_0>0$ is a universal constant and $\beta_0$ is a constant depending on $M_0$. 
\end{remark}
%------------------------------

\appendix 

%=======================================
\section{The estimates of the inhomogeneous term}\label{sec:est-inhom}
%=======================================

%=======================================
\subsection{Preliminary calculation}\label{subsec:pre-cal}
%=======================================
Recall \eqref{def_Xi} and set $\Phi_i(p)\,(i=1,2)$ as 
\[
\Phi_1(p)=\frac{1}{(1+p^2)^2}-1,\quad \Phi_2(p)=\frac{3p}{(1+p^2)^3}.
\]
Taking account of $\Phi_i(0)=0\,(i=1,2)$, we have 
\[
\Phi_i(p)=\int_0^p\Phi_i'(s)\,ds.
\]
This implies that 
\[
|\Phi_i(p)|\le\int_0^{|p|}|\Phi_i'(s)|\,ds\le\bigl(\sup_{|s|\le m}|\Phi_i'(s)|\bigr)|p|.
\]
Since 
\begin{align*}
\Phi_1'(p)=-\frac{4p}{(1+p^2)^3}, \quad \Phi_2'(p)=\frac{3(1-5p^2)}{(1+p^2)^4}, 
\end{align*}
it follows that for $|p|\le m$
\begin{equation}\label{bound-Phi}
|\Phi_1(p)|\le4m^2, \quad |\Phi_2(p)|\le3(1+5m^2)m.
\end{equation}
Furthermore, for $p(x),p(y)\,(x\ne y)$, we have 
\[
\Phi_i(p(x))-\Phi_i(p(y))=\int_0^{p(x)}\Phi_i'(s)\,ds-\int_0^{p(y)}\Phi_i'(s)\,ds=\int_{p(y)}^{p(x)}\Phi_i'(s)\,ds. 
\]
This yields that 
\[
|\Phi_i(p(x))-\Phi_i(p(y))|\le\int_{\min\{p(x),\,p(y)\}}^{\max\{p(x),\,p(y)\}}|\Phi_i'(s)|\,ds
\le\bigl(\sup_{|s|\le m}|\Phi_i'(s)|\bigr)|p(x)-p(y)|,
\]
so that we get
\begin{equation}\label{Holder-Phi}
\left\{\begin{array}{l}
|\Phi_1(p(x))-\Phi_1(p(y))|\le4m|p(x)-p(y)|, \\
|\Phi_2(p(x))-\Phi_2(p(y))|\le3(1+5m^2)|p(x)-p(y)|. 
\end{array}\right.
\end{equation}
The same calculation leads us to
\begin{equation}\label{differ-Phi}
\left\{\begin{array}{l}
|\Phi_1(p^{(1)})-\Phi_1(p^{(2)})|\le4m|p^{(1)}-p^{(2)}|, \\
|\Phi_2(p^{(1)})-\Phi_2(p^{(2)})|\le3(1+5m^2)|p^{(1)}-p^{(2)}|. 
\end{array}\right.
\end{equation}
In addition, we observe that for $p^{(j)}(x),p^{(j)}(y)\,(x\ne y,j=1,2)$
\begin{align*}
&\hspace*{-8pt}
\Phi_i(p^{(1)}(x))-\Phi_i(p^{(2)}(x))-\{\Phi_i(p^{(1)}(y))-\Phi_i(p^{(2)}(y))\} \\
=&\,\int_{p^{(2)}(x)}^{p^{(1)}(x)}\Phi_i'(s)\,ds-\int_{p^{(2)}(y)}^{p^{(1)}(y)}\Phi_i'(s)\,ds \\
=&\,\int_0^{p^{(1)}(x)-p^{(2)}(x)}\Phi_i'(\sigma+p^{(2)}(x))\,d\sigma-\int_0^{p^{(1)}(y)-p^{(2)}(y)}\Phi_i'(\sigma+p^{(2)}(y))\,d\sigma \\
%=&\,\int_0^{p^{(1)}(x)-p^{(2)}(x)}\Phi_i'(\sigma+p^{(2)}(x))\,d\sigma-\int_0^{p^{(1)}(y)-p^{(2)}(y)}\Phi_i'(\sigma+p^{(2)}(x))\,d\sigma \\
%&\,+\int_0^{p^{(1)}(y)-p^{(2)}(y)}\Phi_i'(\sigma+p^{(2)}(x))\,d\sigma-\int_0^{p^{(1)}(y)-p^{(2)}(y)}\Phi_i'(\sigma+p^{(2)}(y))\,d\sigma \\
=&\,\int_{p^{(1)}(y)-p^{(2)}(y)}^{p^{(1)}(x)-p^{(2)}(x)}\Phi_i'(\sigma+p^{(2)}(x))\,d\sigma \\
&\,+\int_0^{p^{(1)}(y)-p^{(2)}(y)}\{\Phi_i'(\sigma+p^{(2)}(x))-\Phi_i'(\sigma+p^{(2)}(y))\}\,d\sigma.
\end{align*}
It follows from the mean value theorem that for $\theta\in(0,1)$ 
\begin{align*}
\Phi_i'(\sigma+p^{(2)}(x))-\Phi_i'(\sigma+p^{(2)}(y))
=\Phi_i''(\sigma+\theta p^{(2)}(x)+(1-\theta)p^{(2)}(y))(p^{(2)}(x)-p^{(2)}(y)), 
\end{align*}
so that we obtain 
\begin{align*}
&\hspace*{-8pt}
|\Phi_i(p^{(1)}(x))-\Phi_i(p^{(2)}(x))-\{\Phi_i(p^{(1)}(y))-\Phi_i(p^{(2)}(y))\}| \\
=&\,\bigl(\sup_{|s|\le4m}|\Phi_i'(s)|\bigr)|p^{(1)}(x)-p^{(2)}(x)-(p^{(1)}(y)-p^{(2)}(y))| \\
&\,+\bigl(\sup_{|s|\le4m}|\Phi_i''(s)|\bigr)|p^{(2)}(x)-p^{(2)}(y)||p^{(1)}(y)-p^{(2)}(y)|.
\end{align*}
Since 
\[
\Phi_1''(s)=-\frac{4(1-5s^2)}{(1+s^2)^4}, \quad \Phi_2''(s)=-\frac{18s(3-5s^2)}{(1+s^2)^5},
\]
we see that for $|s|\le4m$
\begin{align}
&\hspace*{-8pt}
|\Phi_1(p^{(1)}(x))-\Phi_1(p^{(2)}(x))-\{\Phi_1(p^{(1)}(y))-\Phi_1(p^{(2)}(y))\}| \label{differ-Holder-Phi-1} \\
=&\,16m|p^{(1)}(x)-p^{(2)}(x)-(p^{(1)}(y)-p^{(2)}(y))| \notag \\
&\,+4(1+80m^2)|p^{(2)}(x)-p^{(2)}(y)||p^{(1)}(y)-p^{(2)}(y)|, \notag \\
&\hspace*{-8pt}
|\Phi_2(p^{(1)}(x))-\Phi_2(p^{(2)}(x))-\{\Phi_2(p^{(1)}(y))-\Phi_2(p^{(2)}(y))\}| \label{differ-Holder-Phi-2} \\
=&\,3(1+80m^2)|p^{(1)}(x)-p^{(2)}(x)-(p^{(1)}(y)-p^{(2)}(y))| \notag \\
&\,+72(3+80m^2)m|p^{(2)}(x)-p^{(2)}(y)||p^{(1)}(y)-p^{(2)}(y)|. \notag
\end{align}

%=======================================
\subsection{The estimates of the H\"older norm}\label{subsec:Holder-est-inhom}
%=======================================
Let us derive the estimates of the H\"older norm for the inhomogeneous term 
\[
f^w(x,t)=\Xi(w_x,w_{xx},w_{xxx})(x,t)=\Phi_1(w_x)w_{xxx}-\Phi_2(w_x)w_{xx}^2.
\]
We show the following lemmas.

%------------------------------
\begin{lemma}\label{lem:f-Holder}
Let $\gamma\in(0,1)$. Assume that for $M>0$
\begin{equation}\label{assumption}
\|w\|_{B_{\frac{2+\gamma}4}((0,T];BUC^{3+\gamma}(\mathbb{R}_+))}\le M.
\end{equation}
Then there exists $C_M>0$ such that 
\[
\|f^w\|_{B_{\frac{2+\gamma}4 }((0,T];BUC^{\gamma}(\mathbb{R}_+))}\le C_M\|w_x\|_{B([0,T];BUC(\mathbb{R}_+)}^2. 
%\le CM(1+\|w_x\|_{B([0,T];BUC(\mathbb{R}_+)}^2)\|w_x\|_{B([0,T];BUC(\mathbb{R}_+)}^2. 
\]
\end{lemma}
%------------------------------

\noindent
Note that the interpolation inequalities 
\begin{align}
\label{interpolation-1}
&\|w_{xx}\|_\infty\le c_1\|w_x\|_\infty^{\frac12}\|w_{xxx}\|_\infty^{\frac12}, \\
\label{interpolation-2}
&\|w_{xxx}\|_\infty\le c_2\|w_x\|_\infty^{\frac{\gamma}{2+\gamma}}[w_{xxx}]_\gamma^{\frac2{2+\gamma}}, \quad 
[w_x]_\gamma\le c_3\|w_x\|_\infty^{\frac2{2+\gamma}}[w_{xxx}]_\gamma^{\frac{\gamma}{2+\gamma}}, \\
\label{interpolation-3}
&\|w_{xx}\|_\infty\le c_4\|w_x\|_\infty^{\frac{1+\gamma}{2+\gamma}}[w_{xxx}]_\gamma^{\frac1{2+\gamma}}, \quad
[w_{xx}]_\gamma\le c_5\|w_x\|_\infty^{\frac1{2+\gamma}}[w_{xxx}]_\gamma^{\frac{1+\gamma}{2+\gamma}}.
\end{align}
are given by Tanabe \cite[Theorem 3.1]{tan;97;book}. 

\begin{proof}
As for $\|f^w(t)\|_\infty$, it follows from \eqref{bound-Phi} and \eqref{assumption} that 
\begin{align*}
\|\Phi_1(w_x(t))w_{xxx}(t)\|_\infty
\le\|\Phi_1(w_x(t))\|_\infty\|w_{xxx}(t)\|_\infty
\le4M\|w_x(t)\|_\infty^2T^{\frac{\gamma}4}t^{-\frac{2+\gamma}4}, 
\end{align*}
and from \eqref{bound-Phi}, \eqref{assumption} and \eqref{interpolation-1} that 
\begin{align*}
\|\Phi_2(w_x(t))w_{xx}^2(t)\|_\infty
\le&\,\|\Phi_2(w_x(t))\|_\infty\|w_{xx}(t)\|_\infty^2 \\
\le&\,c_1^2\|\Phi_2(w_x(t))\|_\infty\|w_x(t)\|_\infty\|w_{xxx}(t)\|_\infty \\
\le&\,3c_1^2M(1+5\|w_x(t)\|_\infty^2)\|w_x(t)\|_\infty^2T^{\frac{\gamma}4}t^{-\frac{2+\gamma}4}.
\end{align*}
Thus there exists $C_M>0$ such that 
\[
\frac1{T^{\frac{\gamma}4}}\sup_{t\in(0,T]}t^{\frac{2+\gamma}4}\|f^w(t)\|_\infty\le C_M\|w_x(t)\|_\infty^2. 
%\le CM(1+\|w_x\|_{B([0,T];BUC(\mathbb{R}_+)}^2)\|w_x(t)\|_\infty^2. 
\]
Concerning $[f^w(t)]_\gamma$, \eqref{bound-Phi}, \eqref{Holder-Phi}, \eqref{assumption} and \eqref{interpolation-2} imply that 
\begin{align*}
[\Phi_1(w_x(t))w_{xxx}(t)]_\gamma
\le&\,[\Phi_1(w_x(t))]_\gamma\|w_{xxx}(t)\|_\infty+\|\Phi_1(w_x(t))\|_\infty[w_{xxx}(t)]_\gamma \\
\le&\,4\|w_x(t)\|_\infty[w_x(t)]_\gamma\|w_{xxx}(t)\|_\infty+4\|w_x(t)\|_\infty^2[w_{xxx}(t)]_\gamma \\
\le&\,4c_2c_3\|w_x(t)\|_\infty\|w_x(t)\|_\infty[w_{xxx}(t)]_\gamma+4\|w_x(t)\|_\infty^2[w_{xxx}(t)]_\gamma \\
\le&\,4(c_2c_3+1)M\|w_x(t)\|_\infty^2t^{-\frac{2+\gamma}4},
\end{align*}
and \eqref{bound-Phi}, \eqref{Holder-Phi} and \eqref{assumption}--\eqref{interpolation-3} yield that 
\begin{align*}
[\Phi_2(w_x(t))w_{xx}^2(t)]_\gamma
\le&\,[\Phi_2(w_x(t))]_\gamma\|w_{xx}(t)\|^2+\|\Phi_2(w_x(t))\|_\infty[w_{xx}^2(t)]_\gamma \\
\le&\,3c_1^2(1+5\|w_x(t)\|_\infty^2)[w_x(t)]_\gamma\|w_x(t)\|_\infty\|w_{xxx}(t)\|_\infty \\
&\,+6(1+5\|w_x(t)\|_\infty^2)\|w_x(t)\|_\infty\|w_{xx}(t)\|_\infty[w_{xx}(t)]_\gamma \\
\le&\,3c_1^2c_2c_3(1+5\|w_x(t)\|_\infty^2)\|w_x(t)\|_\infty^2[w_{xxx}(t)]_\gamma \\
&\,+6c_4c_5(1+5\|w_x(t)\|_\infty^2)\|w_x(t)\|_\infty\|w_x(t)\|_\infty[w_{xxx}(t)]_\gamma \\
\le&\,3(c_1^2c_2c_3+2c_4c_5)M(1+5\|w_x(t)\|_\infty^2)\|w_x(t)\|_\infty^2t^{-\frac{2+\gamma}4}.
\end{align*}
Hence there exists $C_M>0$ such that 
\[
\sup_{t\in(0,T]}t^{\frac{2+\gamma}4}[f^w(t)]_\gamma
\le C_M\|w_x(t)\|_\infty^2. 
%\le CM(1+\|w_x\|_{B([0,T];BUC(\mathbb{R}_+)}^2)\|w_x(t)\|_\infty^2. 
\]
Consequently, we have the desired result. 
\end{proof}

\medskip
%------------------------------
\begin{lemma}\label{lem:differ-f-Holder}
Let $\gamma\in(0,1)$ and let $w^{(j)}(j=1,2)$ satisfy \eqref{assumption}. 
Then there exists $C_M>0$ such that 
\begin{align*}
&\|f^{w^{(1)}}-f^{w^{(2)}}\|_{B_{\frac{2+\gamma}4 }((0,T];BUC^{\gamma}(\mathbb{R}_+))} \\
&\le C_M\bigl(\max_{j=1,2}\|w^{(j)}_x\|_{B([0,T];BUC(\mathbb{R}_+)}\bigr)
\|w^{(1)}-w^{(2)}\|_{B_{\frac{2+\gamma}4}((0,T];BUC^{3+\gamma}(\mathbb{R}_+))}.
%&\le C(m+M)\bigl(\max_{j=1,2}\|w^{(j)}_x\|_{B([0,T];BUC(\mathbb{R}_+)}\bigr)
%\|w^{(1)}-w^{(2)}\|_{B_{\frac{2+\gamma}4}((0,T];BUC^{3+\gamma}(\mathbb{R}_+))}.
\end{align*}
\end{lemma}
%------------------------------

\begin{proof}
For $w^{(j)}(j=1,2)$, we first observe that 
\begin{align*}
%&\hspace*{-8pt}
f^{w^{(1)}}-f^{w^{(2)}} %\\
%=&\,\Xi(w^{(1)}_x,w^{(1)}_{xx},w^{(1)}_{xxx})-\Xi(w^{(2)}_x,w^{(2)}_{xx},w^{(2)}_{xxx}) \\
%=&\,\Phi_1(w^{(1)}_x)w^{(1)}_{xxx}-\Phi_2(w^{(1)}_x)(w^{(1)}_{xx})^2
%-\{\Phi_1(w^{(2)}_x)w^{(2)}_{xxx}-\Phi_2(w^{(2)}_x)(w^{(2)}_{xx})^2\} \\
%=&\,(\Phi_1(w^{(1)}_x)-\Phi_1(w^{(2)}_x))w^{(1)}_{xxx}+\Phi_1(w^{(2)}_x)(w^{(1)}_{xxx}-w^{(2)}_{xxx}) \\
%&\,+(\Phi_2(w^{(1)}_x)-\Phi_2(w^{(2)}_x))(w^{(1)}_{xx})^2+\Phi_2(w^{(2)}_x)\{(w^{(1)}_{xx})^2-(w^{(2)}_{xx})^2\} \\
=\,\,&\,(\Phi_1(w^{(1)}_x)-\Phi_1(w^{(2)}_x))w^{(1)}_{xxx}+\Phi_1(w^{(2)}_x)(w^{(1)}_{xxx}-w^{(2)}_{xxx}) \\
&\,+(\Phi_2(w^{(1)}_x)-\Phi_2(w^{(2)}_x))(w^{(1)}_{xx})^2+\Phi_2(w^{(2)}_x)(w^{(1)}_{xx}+w^{(2)}_{xx})(w^{(1)}_{xx}-w^{(2)}_{xx}) \\
=:&\,\Psi_1+\Psi_2.
\end{align*}
Set $m:=\max_{j=1,2}\|w^{(j)}_x\|_{B([0,T];BUC(\mathbb{R}_+)}$. 
As regards $\|f^{w^{(1)}}(t)-f^{w^{(2)}}(t)\|_\infty$, \eqref{differ-Phi} and \eqref{assumption} imply that 
\begin{align*}
&\hspace*{-8pt}
\|\Psi_1(t)\|_\infty \\
\le&\,\|\Phi_1(w^{(1)}_x(t))-\Phi_1(w^{(2)}_x(t))\|_\infty\|w^{(1)}_{xxx}(t)\|_\infty
+\|\Phi_1(w^{(2)}_x(t))\|_\infty\|w^{(1)}_{xxx}(t)-w^{(2)}_{xxx}(t)\|_\infty \\
\le&\,4m\|w^{(1)}_x(t)-w^{(2)}_x(t)\|_\infty\|w^{(1)}_{xxx}(t)\|_\infty+4m^2\|w^{(1)}_{xxx}(t)-w^{(2)}_{xxx}(t)\|_\infty \\
\le&\,4mM\|w^{(1)}_x(t)-w^{(2)}_x(t)\|_\infty T^{\frac{\gamma}4}t^{-\frac{2+\gamma}4} \\
&\,+4m^2\Bigl(\frac1{T^{\frac{\gamma}4}}t^{\frac{2+\gamma}4}\|w^{(1)}_{xxx}(t)-w^{(2)}_{xxx}(t)\|_\infty\Bigr)
T^{\frac{\gamma}4}t^{-\frac{2+\gamma}4},
\end{align*}
and \eqref{differ-Phi}, \eqref{assumption} and \eqref{interpolation-1} yield that 
\begin{align*}
&\hspace*{-8pt}
\|\Psi_2(t)\|_\infty \\
\le&\,\|\Phi_2(w^{(1)}_x(t))-\Phi_2(w^{(2)}_x(t))\|_\infty\|w^{(1)}_{xx}(t)\|_\infty^2 \\
&\,+\|\Phi_2(w^{(2)}_x(t))\|_\infty(\|w^{(1)}_{xx}(t)\|_\infty+\|w^{(2)}_{xx}(t)\|_\infty)\|w^{(1)}_{xx}(t)-w^{(2)}_{xx}(t)\|_\infty \\
\le&\,3c_1(1+5m^2)\|w^{(1)}_x(t)-w^{(2)}_x(t)\|_\infty\|w^{(1)}_x(t)\|_\infty\|w^{(1)}_{xxx}(t)\|_\infty \\
&\,+3c_1^2(1+5m^2)m(\|w^{(1)}_x(t)\|_\infty^{\frac12}\|w^{(1)}_{xxx}(t)\|_\infty^{\frac12}
+\|w^{(2)}_x(t)\|_\infty^{\frac12}\|w^{(2)}_{xxx}(t)\|_\infty^{\frac12}) \\
&\hspace*{12.5pt}
\|w^{(1)}_x(t)-w^{(2)}_x(t)\|_\infty^{\frac12}\|w^{(1)}_{xxx}(t)-w^{(2)}_{xxx}(t)\|_\infty^{\frac12} \\
\le&\,3c_1(1+5m^2)mM\|w^{(1)}_x(t)-w^{(2)}_x(t)\|_\infty T^{\frac{\gamma}4}t^{-\frac{2+\gamma}4} \\
&\,+6c_1^2(1+5m^2)m^{\frac32}M^{\frac12} \|w^{(1)}_x(t)-w^{(2)}_x(t)\|_\infty^{\frac12} \\
&\hspace*{12.5pt}
\Bigl(\frac1{T^{\frac{\gamma}4}}t^{\frac{2+\gamma}4}\|w^{(1)}_{xxx}(t)-w^{(2)}_{xxx}(t)\|_\infty\Bigr)^{\frac12}
T^{\frac{\gamma}4}t^{-\frac{2+\gamma}4} \\
\le&\,3c_1(1+5m^2)mM\|w^{(1)}_x(t)-w^{(2)}_x(t)\|_\infty T^{\frac{\gamma}4}t^{-\frac{2+\gamma}4} \\
&\,+3c_1^2(1+5m^2)\Bigl\{mM\|w^{(1)}_x(t)-w^{(2)}_x(t)\|_\infty \\
&\hspace*{12.5pt}
+m^2\Bigl(\frac1{T^{\frac{\gamma}4}}t^{\frac{2+\gamma}4}\|w^{(1)}_{xxx}(t)-w^{(2)}_{xxx}(t)\|_\infty\Bigr)\Bigr\}
T^{\frac{\gamma}4}t^{-\frac{2+\gamma}4} \\
\le&\,3c_1(1+c_1)(1+5m^2)mM\|w^{(1)}_x(t)-w^{(2)}_x(t)\|_\infty T^{\frac{\gamma}4}t^{-\frac{2+\gamma}4} \\
&\,+3c_1^2(1+5m^2)m^2\Bigl(\frac1{T^{\frac{\gamma}4}}t^{\frac{2+\gamma}4}\|w^{(1)}_{xxx}(t)-w^{(2)}_{xxx}(t)\|_\infty\Bigr)
T^{\frac{\gamma}4}t^{-\frac{2+\gamma}4}.
\end{align*}
Therefore, there exists $C_M>0$ such that 
\begin{align*}
\frac1{T^{\frac{\gamma}4}}\sup_{t\in(0,T]}t^{\frac{2+\gamma}4}\|f^{w^{(1)}}(t)-f^{w^{(2)}}(t)\|_\infty %\\
\le C_Mm\|w^{(1)}-w^{(2)}\|_{B_{\frac{2+\gamma}4}((0,T];BUC^{3+\gamma}(\mathbb{R}_+))}.
\end{align*}
Concerning $[f^{w^{(1)}}(t)-f^{w^{(2)}}(t)]_\gamma$, \eqref{Holder-Phi}--\eqref{differ-Holder-Phi-1}, \eqref{assumption}, 
\eqref{interpolation-2} and \eqref{interpolation-3} imply that 
\begin{align*}
&\hspace*{-8pt}
[\Psi_1(t)]_\gamma \\
\le&\,[\Phi_1(w^{(1)}_x(t))-\Phi_1(w^{(2)}_x(t))]_\gamma\|w^{(1)}_{xxx}\|_\infty
+\|\Phi_1(w^{(1)}_x(t))-\Phi_1(w^{(2)}_x(t))\|_\infty[w^{(1)}_{xxx}(t)]_\gamma \\
&\,+[\Phi_1(w^{(2)}_x(t))]_\gamma\|w^{(1)}_{xxx}(t)-w^{(2)}_{xxx}(t)\|_\infty
+\|\Phi_1(w^{(2)}_x(t))\|_\infty[w^{(1)}_{xxx}(t)-w^{(2)}_{xxx}(t)]_\gamma \\
\le&\,\bigl\{16m[w^{(1)}_x(t)-w^{(2)}_x(t)]_\gamma+4(1+80m^2)[w^{(2)}_x(t)]_\gamma\|w^{(1)}_x(t)-w^{(2)}_x(t)\|_\infty\bigr\}\|w^{(1)}_{xxx}(t)\|_\infty \\
&\,+4m\|w^{(1)}_x(t)-w^{(2)}_x(t)\|_\infty[w^{(1)}_{xxx}(t)]_\gamma+4m[w^{(2)}_x(t)]_\gamma\|w^{(1)}_{xxx}(t)-w^{(2)}_{xxx}(t)\|_\infty \\
&\,+4m^2[w^{(1)}_{xxx}(t)-w^{(2)}_{xxx}(t)]_\gamma \\
\le&\,4c_2c_3\bigl\{4m\|w^{(1)}_x(t)-w^{(2)}_x(t)\|_\infty^{\frac2{2+\gamma}}
[w^{(1)}_{xxx}(t)-w^{(2)}_{xxx}(t)]_\gamma^{\frac{\gamma}{2+\gamma}} \\
&\,+(1+80m^2)\|w^{(2)}_x(t)\|_\infty^{\frac2{2+\gamma}}[w^{(2)}_{xxx}(t)]_\gamma^{\frac{\gamma}{2+\gamma}}\|w^{(1)}_x(t)-w^{(2)}_x(t)\|_\infty\bigr\} \\
&\hspace*{12.5pt}
\|w^{(1)}_x(t)\|_\infty^{\frac{\gamma}{2+\gamma}}[w^{(1)}_{xxx}(t)]_\gamma^{\frac2{2+\gamma}} \\ 
&\,+4m\|w^{(1)}_x(t)-w^{(2)}_x(t)\|_\infty[w^{(1)}_{xxx}(t)]_\gamma \\
&\,+4c_2c_3m\|w^{(2)}_x(t)\|_\infty^{\frac2{2+\gamma}}[w^{(2)}_{xxx}(t)]_\gamma^{\frac{\gamma}{2+\gamma}}
\|w^{(1)}_x(t)-w^{(2)}_x(t)\|_\infty^{\frac{\gamma}{2+\gamma}}[w^{(1)}_{xxx}(t)-w^{(2)}_{xxx}(t)]_\gamma^{\frac2{2+\gamma}} \\
&\,+4m^2[w^{(1)}_{xxx}(t)-w^{(2)}_{xxx}(t)]_\gamma \\
\le&\,20c_2c_3m\bigl(M\|w^{(1)}_x(t)-w^{(2)}_x(t)\|_\infty\bigr)^{\frac2{2+\gamma}}
\bigl\{m(t^{\frac{2+\gamma}4}[w^{(1)}_{xxx}(t)-w^{(2)}_{xxx}(t)]_\gamma)\bigr\}^{\frac{\gamma}{2+\gamma}}t^{-\frac{2+\gamma}4} \\
&\,+4c_2c_3(1+80m^2)mM\|w^{(1)}_x(t)-w^{(2)}_x(t)\|_\infty t^{-\frac{2+\gamma}4} \\
&\,+4mM\|w^{(1)}_x(t)-w^{(2)}_x(t)\|_\infty t^{-\frac{2+\gamma}4} \\
&\,+4m^2(t^{\frac{2+\gamma}4}[w^{(1)}_{xxx}(t)-w^{(2)}_{xxx}(t)]_\gamma)t^{-\frac{2+\gamma}4} \\
\le&\,20c_2c_3c_6m\bigl\{M\|w^{(1)}_x(t)-w^{(2)}_x(t)\|_\infty
+m(t^{\frac{2+\gamma}4}[w^{(1)}_{xxx}(t)-w^{(2)}_{xxx}(t)]_\gamma)\bigr\}t^{-\frac{2+\gamma}4} \\
&\,+4(c_2c_3(1+80m^2)+1)mM\|w^{(1)}_x(t)-w^{(2)}_x(t)\|_\infty t^{-\frac{2+\gamma}4} \\
&\,+4m^2(t^{\frac{2+\gamma}4}[w^{(1)}_{xxx}(t)-w^{(2)}_{xxx}(t)]_\gamma)t^{-\frac{2+\gamma}4},
\end{align*}
and \eqref{Holder-Phi}, \eqref{differ-Phi}, \eqref{differ-Holder-Phi-2}, \eqref{assumption}, \eqref{interpolation-2} and 
\eqref{interpolation-3} yield that 
\begin{align*}
&\hspace*{-8pt}
[\Psi_2(t)]_\gamma \\
\le&\,[\Phi_2(w^{(1)}_x(t))-\Phi_2(w^{(2)}_x(t))]_\gamma\|w^{(1)}_{xx}(t)\|_\infty^2 \\
&\,+\|\Phi_2(w^{(1)}_x(t))-\Phi_2(w^{(2)}_x(t))\|_\infty[(w^{(1)}_{xx}(t))^2]_\gamma \\
&\,+[\Phi_2(w^{(2)}_x(t))]_\gamma(\|w^{(1)}_{xx}(t)\|_\infty+\|w^{(2)}_{xx}(t)\|_\infty)\|w^{(1)}(t)_{xx}-w^{(2)}_{xx}(t)\|_\infty \\
&\,+\|\Phi_2(w^{(2)}_x(t))\|_\infty[(w^{(1)}_{xx}(t))^2-(w^{(2)}_{xx}(t))^2]_\gamma \\
\le&\,\bigl\{3(1+80m^2)[w^{(1)}_x(t)-w^{(2)}_x(t)]_\gamma \\
&\,+72(3+80m^2)m[w^{(2)}_x(t)]_\gamma\|w^{(1)}_x(t)-w^{(2)}_x(t)\|_\infty\bigr\}\|w^{(1)}_{xx}(t)\|_\infty^2 \\
&\,+6(1+5m^2)\|w^{(1)}_x(t)-w^{(2)}_x(t)\|_\infty\|w^{(1)}_{xx}(t)\|_\infty[w^{(1)}_{xx}(t)]_\gamma \\
&\,+3(1+5m^2)[w^{(2)}_x(t)]_\gamma(\|w^{(1)}_{xx}(t)\|_\infty+\|w^{(2)}_{xx}(t)\|_\infty)\|w^{(1)}_{xx}(t)-w^{(2)}_{xx}(t)\|_\infty \\
&\,+3(1+5m^2)m(\|w^{(1)}_{xx}(t)\|_\infty+\|w^{(2)}_{xx}(t)\|_\infty)[w^{(1)}_{xx}(t)-w^{(2)}_{xx}(t)]_\gamma \\
&\,+3(1+5m^2)m([w^{(1)}_{xx}(t)]_\gamma+[w^{(2)}_{xx}(t)]_\gamma)\|w^{(1)}_{xx}(t)-w^{(2)}_{xx}(t)\|_\infty \\
\le&\,3c_3c_4^2\bigl\{(1+80m^2)\|w^{(1)}_x(t)-w^{(2)}_x(t)\|_\infty^{\frac2{2+\gamma}}
[w^{(1)}_{xxx}(t)-w^{(2)}_{xxx}(t)]_\gamma^{\frac{\gamma}{2+\gamma}} \\
&\,+24(3+80\|w_x(t)\|_\infty^2)m\|w^{(2)}_x(t)\|_\infty^{\frac2{2+\gamma}}[w^{(2)}_{xxx}(t)]_\gamma^{\frac{\gamma}{2+\gamma}}
\|w^{(1)}_x(t)-w^{(2)}_x(t)\|_\infty\bigr\} \\
&\hspace*{12.5pt}
\bigl(\|w^{(1)}_x(t)\|_\infty^{\frac{1+\gamma}{2+\gamma}}[w^{(1)}_{xxx}(t)]_\gamma^{\frac1{2+\gamma}}\bigr)^2 \\
&\,+6c_4c_5(1+5m^2)\|w^{(1)}_x(t)-w^{(2)}_x(t)\|_\infty\|w^{(1)}_x(t)\|_\infty^{\frac{1+\gamma}{2+\gamma}} \\
&\hspace*{12.5pt}
[w^{(1)}_{xxx}(t)]_\gamma^{\frac1{2+\gamma}}\|w^{(1)}_x(t)\|_\infty^{\frac1{2+\gamma}}[w^{(1)}_{xxx}(t)]_\gamma^{\frac{1+\gamma}{2+\gamma}} \\
&\,+3c_3c_4^2(1+5m^2)\|w^{(1)}_x(t)\|_\infty^{\frac2{2+\gamma}}[w^{(1)}_{xxx}(t)]_\gamma^{\frac{\gamma}{2+\gamma}} \\
&\hspace*{12.5pt}
(\|w^{(1)}_x(t)\|_\infty^{\frac{1+\gamma}{2+\gamma}}[w^{(1)}_{xxx}(t)]_\gamma^{\frac1{2+\gamma}}
+\|w^{(2)}_x(t)\|_\infty^{\frac{1+\gamma}{2+\gamma}}[w^{(2)}_{xxx}(t)]_\gamma^{\frac1{2+\gamma}}) \\
&\hspace*{12.5pt}
\|w^{(1)}_x(t)-w^{(1)}_x(t)\|_\infty^{\frac{1+\gamma}{2+\gamma}}[w^{(1)}_{xxx}(t)-w^{(2)}_{xxx}(t)]_\gamma^{\frac1{2+\gamma}} \\
&\,+3c_4c_5(1+5m^2)\|w^{(2)}_x(t)\|_\infty \\
&\hspace*{12.5pt}
(\|w^{(1)}_x(t)\|_\infty^{\frac{1+\gamma}{2+\gamma}}[w^{(1)}_{xxx}(t)]_\gamma^{\frac1{2+\gamma}}
+\|w^{(2)}_x(t)\|_\infty^{\frac{1+\gamma}{2+\gamma}}[w^{(2)}_{xxx}(t)]_\gamma^{\frac1{2+\gamma}}) \\
&\hspace*{12.5pt}
\|w^{(1)}_x(t)-w^{(1)}_x(t)\|_\infty^{\frac1{2+\gamma}}[w^{(1)}_{xxx}(t)-w^{(2)}_{xxx}(t)]_\gamma^{\frac{1+\gamma}{2+\gamma}} \\
&\,+3c_4c_5(1+5m^2)\|w^{(2)}_x(t)\|_\infty \\
&\hspace*{12.5pt}
(\|w^{(1)}_x(t)\|_\infty^{\frac1{2+\gamma}}[w^{(1)}_{xxx}(t)]_\gamma^{\frac{1+\gamma}{2+\gamma}}
+\|w^{(2)}_x(t)\|_\infty^{\frac1{2+\gamma}}[w^{(2)}_{xxx}(t)]_\gamma^{\frac{1+\gamma}{2+\gamma}}) \\
&\hspace*{12.5pt}
\|w^{(1)}_x(t)-w^{(1)}_x(t)\|_\infty^{\frac{1+\gamma}{2+\gamma}}[w^{(1)}_{xxx}(t)-w^{(2)}_{xxx}(t)]_\gamma^{\frac1{2+\gamma}} \\
%\le&\,3c_3c_4^2(1+80m^2)m\bigl(M\|w^{(1)}_x(t)-w^{(2)}_x(t)\|_\infty\bigr)^{\frac2{2+\gamma}}
%\bigl\{m(t^{\frac{2+\gamma}4}[w^{(1)}_{xxx}(t)-w^{(2)}_{xxx}(t)]_\gamma)\bigr\}^{\frac{\gamma}{2+\gamma}}t^{-\frac{2+\gamma}4} \\
%&\,+72c_3c_4^2(3+80m^2)m^3M\|w^{(1)}_x(t)-w^{(2)}_x(t)\|_\infty t^{-\frac{2+\gamma}4} \\
%&\,+6c_4c_5(1+5m^2)mM\|w^{(1)}_x(t)-w^{(2)}_x(t)\|_\infty t^{-\frac{2+\gamma}4} \\
%&\,+6c_3c_4^2(1+5m^2)m\bigl(M\|w^{(1)}_x(t)-w^{(1)}_x(t)\|_\infty\bigr)^{\frac{1+\gamma}{2+\gamma}}
%\bigl\{m(t^{\frac{2+\gamma}4}[w^{(1)}_{xxx}(t)-w^{(2)}_{xxx}(t)]_\gamma)\bigr\}^{\frac1{2+\gamma}}t^{-\frac{2+\gamma}4} \\
%&\,+6c_4c_5(1+5m^2)m\bigl(M\|w^{(1)}_x(t)-w^{(1)}_x(t)\|_\infty\bigr)^{\frac1{2+\gamma}}
%\bigl\{m(t^{\frac{2+\gamma}4}[w^{(1)}_{xxx}(t)-w^{(2)}_{xxx}(t)]_\gamma)\bigr\}^{\frac{1+\gamma}{2+\gamma}}t^{-\frac{2+\gamma}4} \\
%&\,+6c_4c_5(1+5m^2)m\bigl(M\|w^{(1)}_x(t)-w^{(1)}_x(t)\|_\infty\bigr)^{\frac{1+\gamma}{2+\gamma}}
%\bigl\{m(t^{\frac{2+\gamma}4}[w^{(1)}_{xxx}(t)-w^{(2)}_{xxx}(t)]_\gamma)\bigr\}^{\frac1{2+\gamma}}t^{-\frac{2+\gamma}4} \\
\le&\,3c_3c_4^2(1+80m^2)m\bigl(M\|w^{(1)}_x(t)-w^{(2)}_x(t)\|_\infty\bigr)^{\frac2{2+\gamma}} \\
&\,\bigl\{m(t^{\frac{2+\gamma}4}[w^{(1)}_{xxx}(t)-w^{(2)}_{xxx}(t)]_\gamma)\bigr\}^{\frac{\gamma}{2+\gamma}}t^{-\frac{2+\gamma}4} \\
&\,+72c_3c_4^2(3+80m^2)m^3M\|w^{(1)}_x(t)-w^{(2)}_x(t)\|_\infty t^{-\frac{2+\gamma}4} \\
&\,+6c_4c_5(1+5m^2)mM\|w^{(1)}_x(t)-w^{(2)}_x(t)\|_\infty t^{-\frac{2+\gamma}4} \\
&\,+6c_4(1+5m^2)m \\
&\hspace*{12.5pt}
\bigl[c_3c_4\bigl(M\|w^{(1)}_x(t)-w^{(1)}_x(t)\|_\infty\bigr)^{\frac{1+\gamma}{2+\gamma}}
\bigl\{m(t^{\frac{2+\gamma}4}[w^{(1)}_{xxx}(t)-w^{(2)}_{xxx}(t)]_\gamma)\bigr\}^{\frac1{2+\gamma}} \\
&\,\hspace*{12.5pt}
+c_5\bigl(M\|w^{(1)}_x(t)-w^{(1)}_x(t)\|_\infty\bigr)^{\frac1{2+\gamma}}
\bigl\{m(t^{\frac{2+\gamma}4}[w^{(1)}_{xxx}(t)-w^{(2)}_{xxx}(t)]_\gamma)\bigr\}^{\frac{1+\gamma}{2+\gamma}} \\
&\,\hspace*{12.5pt}
+c_5\bigl(M\|w^{(1)}_x(t)-w^{(1)}_x(t)\|_\infty\bigr)^{\frac{1+\gamma}{2+\gamma}}
\bigl\{m(t^{\frac{2+\gamma}4}[w^{(1)}_{xxx}(t)-w^{(2)}_{xxx}(t)]_\gamma)\bigr\}^{\frac1{2+\gamma}}\bigr]t^{-\frac{2+\gamma}4} \\
\le&\,3c_3c_4^2c_6(1+80m^2)m \\
&\,\bigl\{M\|w^{(1)}_x(t)-w^{(2)}_x(t)\|_\infty+m(t^{\frac{2+\gamma}4}[w^{(1)}_{xxx}(t)-w^{(2)}_{xxx}(t)]_\gamma)\bigr\}t^{-\frac{2+\gamma}4} \\
&\,+72c_3c_4^2(3+80m^2)m^3M\|w^{(1)}_x(t)-w^{(2)}_x(t)\|_\infty t^{-\frac{2+\gamma}4} \\
&\,+6c_4c_5(1+5m^2)mM\|w^{(1)}_x(t)-w^{(2)}_x(t)\|_\infty t^{-\frac{2+\gamma}4} \\
&\,+6c_4c_6(c_3c_4+2c_5)(1+5m^2)m \\
&\hspace*{12.5pt}
\bigl\{M\|w^{(1)}_x(t)-w^{(1)}_x(t)\|_\infty+m(t^{\frac{2+\gamma}4}[w^{(1)}_{xxx}(t)-w^{(2)}_{xxx}(t)]_\gamma)\bigr\}t^{-\frac{2+\gamma}4}.
\end{align*}
As a result, there exists $C_M>0$ such that 
\[
\sup_{t\in(0,T]}t^{\frac{2+\gamma}4}[f^{w^{(1)}}(t)-f^{w^{(2)}}(t)]_\gamma
\le C_Mm\|w^{(1)}-w^{(2)}\|_{B_{\frac{2+\gamma}4}((0,T];BUC^{3+\gamma}(\mathbb{R}_+))}.
\]
Consequently, we have the desired result. 
\end{proof}

\subsection*{Acknowledgment}
The authors would like to thank Professor Yoshikazu Giga for fruitful discussions on the fundamental solution 
in the half space. 
%In addition, the authors are grateful to the referee for pointing out errors in the original manuscript with insightful comments. 
This project is supported by JSPS KAKENHI Grant Number JP24K06810, JP23H00085, JP23K03215 
for the second author.



\end{document}